\documentclass[11pt]{amsart}
\usepackage[margin=1in]{geometry}
\usepackage{amssymb,amsfonts,amsrefs,amsthm}
\usepackage{color}
\usepackage{verbatim}

\newtheorem{Theorem}{Theorem}[section]
\newtheorem{Proposition}[Theorem]{Proposition} 
\newtheorem{Lemma}[Theorem]{Lemma}
\newtheorem{Corollary}[Theorem]{Corollary}
\newtheorem*{Claim}{Claim}

\theoremstyle{definition}
\newtheorem{Definition}[Theorem]{Definition}

\newtheorem{Conjecture}[Theorem]{Conjecture}
\newtheorem{Question}[Theorem]{Question}

\newcommand{\rcasys}{\mathrm{RCA}_0}
\newcommand{\wklsys}{\mathrm{WKL}_0}
\newcommand{\acasys}{\mathrm{ACA}_0}
\newcommand{\pa}{\mathrm{PA}}

\newcommand{\wklstat}{\mathrm{WKL}}
\newcommand{\wwklstat}{\mathrm{WWKL}}
\newcommand{\dnrstat}{\mathrm{DNR}}
\newcommand{\colstat}{\mathrm{COL}}

\DeclareMathOperator{\dom}{\mathrm{dom}}
\DeclareMathOperator{\ran}{\mathrm{ran}}
\DeclareMathOperator{\leaves}{\mathrm{leaves}}
\DeclareMathOperator{\depth}{\mathrm{depth}}

\newcommand{\andd}{\wedge}
\newcommand{\la}{\langle}
\newcommand{\ra}{\rangle}
\newcommand{\da}{\!\downarrow}
\newcommand{\ua}{\!\uparrow}
\newcommand{\imp}{\rightarrow}

\newcommand{\biimp}{\leftrightarrow}
\newcommand{\Nb}{\mathbb{N}}
\newcommand{\smf}{\smallfrown}
\newcommand{\leqT}{\leq_\mathrm{T}}

\newcommand{\bso}{\mathrm{B}\Sigma^0_1}
\newcommand{\iso}{\mathrm{I}\Sigma^0_1}

\newcommand{\bst}{\mathrm{B}\Sigma^0_2}
\newcommand{\ist}{\mathrm{I}\Sigma^0_2}
\newcommand{\pam}{\pa^{-}}

\newcommand{\MP}[1]{\mathcal{#1}}

\newcommand{\cls}[1]{\boldsymbol{#1}}

\title[Comparing the strength of $\dnrstat$ functions in the absence of $\ist$]{Comparing the strength of diagonally non-recursive functions in the absence of $\Sigma^0_2$ induction}

\author{Fran\c{c}ois G. Dorais}
\address{Department of Mathematics\\
Dartmouth College\\
Hanover NH 03755\\
USA}
\email{francois.g.dorais@dartmouth.edu}
\urladdr{http://math.dartmouth.edu/~dorais/}

\author{Jeffry L. Hirst}
\address{Department of Mathematical Sciences\\
Appalachian State University\\
Boone NC 28608\\
USA}
\email{jlh@math.appstate.edu}
\urladdr{http://mathsci2.appstate.edu/~jlh/}

\author{Paul Shafer}
\address{Department of Mathematics\\
Ghent University\\
Krijgslaan 281 S22\\
B-9000 Ghent\\
Belgium}
\email{paul.shafer@ugent.be}
\urladdr{http://cage.ugent.be/~pshafer/}

\thanks{Paul Shafer is an FWO Pegasus Long Postdoctoral Fellow, and he also acknowledges the support of the Fondation Sciences Math\'ematiques de Paris.
Jeffry Hirst was partially supported by grant (ID\#20800) from the John Templeton Foundation. The opinions expressed in this publication are those of the authors and do not necessarily reflect the views of the John Templeton Foundation.}

\date{\today}

\begin{document}
\begin{abstract}
We prove that the statement ``there is a $k$ such that for every $f$ there is a $k$-bounded diagonally non-recursive function relative to $f$'' does not imply weak K\"onig's lemma over $\rcasys + \bst$.  This answers a question posed by Simpson.  A recursion-theoretic consequence is that the classic fact that every $k$-bounded diagonally non-recursive function computes a $2$-bounded diagonally non-recursive function may fail in the absence of $\ist$.
\end{abstract}

\maketitle

\section{Introduction}
\begin{quote}
\itshape It is a truth universally acknowledged, that a single man in possession of a good $k$-bounded diagonally non-recursive function, must be in want of a $2$-bounded diagonally non-recursive function~\cite{Austen}.
\end{quote}

An enduring project in recursion theory is to determine the amount of induction necessary to prove its classic theorems, particularly those concerning the recursively enumerable sets.  Post's problem and the Friedberg-Muchnik theorem~\cites{Chong:1992uz,Mytilinaios:1989ba,Slaman:1989fc}, the Sacks splitting theorem~\cites{Mytilinaios:1989ba,Slaman:1989fc}, the Sacks density theorem~\cite{Groszek:1996kq}, the infinite injury method~\cites{Chong:1998tu, Chong:1997he, Chong:2001iv}, and even the transitivity of Turing reducibility~\cite{Groszek:1994wa} have all been investigated.  The non-standard methods developed in the course of these studies have been recently applied in \emph{reverse mathematics}, an analysis of the logical strengths of ordinary mathematical statements in the context of second-order arithmetic, and led to solutions of several important open problems in the field.  Remarkably, Chong, Slaman, and Yang proved that stable Ramsey's theorem for pairs is strictly weaker than Ramsey's theorem for pairs~\cite{Chong:2012vp} and that Ramsey's theorem for pairs does not imply induction for $\Sigma^0_2$ predicates~\cite{Chong:2013prep}.  Furthermore, non-standard techniques are necessarily employed in proofs of conservativity results over systems with limited induction, such as the $\Pi^1_1$-conservativities of the cohesive principle and the chain-antichain principle over $\rcasys$ plus bounding for $\Sigma^0_2$ predicates~\cite{Chong:2012kz}.  Similarly, Corduan, Groszek, and Mileti proved what may be described as a dual conservativity result:  an extension of $\rcasys$ by $\Pi^1_1$ axioms proves Ramsey's theorem for singletons on the complete binary tree if and only if the extension proves induction for $\Sigma^0_2$ predicates~\cite{Corduan:2010ig}.  It follows that $\rcasys$ plus bounding for $\Sigma^0_2$ predicates does not prove Ramsey's theorem for singletons on the complete binary tree, which answers a question from~\cite{Chubb:2009gr}.  For a comprehensive introduction to non-standard methods in recursion theory and reverse mathematics, we refer the reader to the recent survey by Chong, Li, and Yang~\cite{Chong:2013wx}.

Within this framework of reverse mathematics, we study the logical strengths of several statements asserting the existence of $k$-bounded diagonally non-recursive functions.  Theorem 5 of Jockusch's classic analysis of diagonally non-recursive functions~\cite{JockuschJr:1989vs} states that every $k$-bounded diagonally non-recursive function computes a $2$-bounded diagonally non-recursive function.  The proof, which Jockusch attributes to Friedberg, is not uniform, and Jockusch proves that this is necessarily the case:  Theorem 6 of~\cite{JockuschJr:1989vs} implies that if $k > 2$ then there is no uniform (i.e., Medvedev) reduction from the class of $k$-bounded diagonally non-recursive functions to the class of $2$-bounded diagonally non-recursive functions.  In a talk given at the 2001 Annual Meeting of the American Philosophical Association~\cite{SimpsonTalk}, Simpson asked if the reduction from $k$-bounded diagonally non-recursive functions to $2$-bounded diagonally non-recursive functions can be implemented $\rcasys$.  Specifically, he asked if the statement ``there is a $k$ such that for every $X$ there a $k$-bounded diagonally non-recursive function relative to $X$'' implies weak K\"onig's lemma over $\rcasys$.  Our main result is that although the statement in question indeed implies weak K\"onig's lemma over $\rcasys$ plus induction for $\Sigma^0_2$ predicates, it does not imply weak K\"onig's lemma over $\rcasys$ plus bounding for $\Sigma^0_2$ predicates.  Consequently, if induction for $\Sigma^0_2$ predicates fails, there may be $k$-bounded diagonally non-recursive functions (for some necessarily non-standard $k$) that do not compute $2$-bounded diagonally non-recursive functions.  This result expresses a sense in which induction for $\Sigma^0_2$ predicates is necessary to prove that every $k$-bounded diagonally non-recursive function computes a $2$-bounded diagonally non-recursive function. 

\section{Background}
We define the fragments of first-order and second-order arithmetic that we consider in this work.  The standard references are H\'ajek and Pudl\'ak's \emph{Metamathematics of First-Order Arithmetic}~\cite{Hajek:1998uh} for fragments of first-order arithmetic and Simpson's \emph{Subsystems of Second Order Arithmetic}~\cite{Simpson:2009vv} for fragments of second-order arithmetic in the context of reverse mathematics.  Reverse mathematics is a foundational program, introduced by Friedman in~\cite{Friedman:1975wy}, dedicated to characterizing the logical strengths of the classic theorems of mathematics when interpreted in second-order arithmetic.  It is thus a fundamentally proof-theoretic endeavor, although its techniques are primarily recursion-theoretic.  We encouragingly refer the interested reader to the introduction of Simpson's book for a hearty introduction to reverse mathematics and its metamathematical motivations.

We pause here to highlight one important notational convention.  As is common when writing about reverse mathematics, throughout this work we use the symbol `$\omega$' to refer to the standard natural numbers and the symbol `$\Nb$' to refer to the first-order part of whatever structure is (often implicitly) under consideration.

\subsection{Fragments of first-order arithmetic}
The basic axioms of Peano arithmetic, here denoted $\pam$, express that $\Nb$ is a discretely ordered commutative semi-ring with $1$. Peano arithmetic, denoted $\pa$, consists of $\pam$ plus the \emph{induction scheme}, which consists of the universal closures of all formulas of the form
\begin{align*}
[\varphi(0) \andd \forall n(\varphi(n) \imp \varphi(n+1))] \imp \forall n \varphi(n).
\end{align*}

Fragments of $\pa$ are obtained by limiting the quantifier complexity of the formulas $\varphi$ allowed in the induction scheme.  For each $n \in \omega$, the $\Sigma^0_n$ ($\Pi^0_n$) induction scheme is the restriction of the induction scheme to $\Sigma^0_n$ ($\Pi^0_n$) formulas $\varphi$, and $\mathrm{I}\Sigma^0_n$ ($\mathrm{I}\Pi^0_n$) denotes the fragment of $\pa$ consisting of $\pam$ plus the $\Sigma^0_n$ ($\Pi^0_n$) induction scheme.  We express induction for $\Delta^0_n$ predicates via the $\Delta^0_n$ induction scheme, which consists of the universal closures of all formulas of the form
\begin{align*}
\forall n (\varphi(n) \biimp \psi(n)) \imp ([\varphi(0) \andd \forall n(\varphi(n) \imp \varphi(n+1))] \imp \forall n \varphi(n)),
\end{align*}
where $\varphi$ is $\Sigma^0_n$ and $\psi$ is $\Pi^0_n$.  The fragment $\mathrm{I}\Delta^0_n$ is then $\pam$ plus the the $\Delta^0_n$ induction scheme.

We also consider fragments of $\pa$ obtained by adding so-called \emph{bounding schemes} (also called \emph{collection schemes}).  The $\Sigma^0_n$ ($\Pi^0_n$) \emph{bounding scheme} consists of the universal closures of all formulas of the form
\begin{align*}
\forall a[(\forall n < a)(\exists m)\varphi(n,m) \imp \exists b(\forall n < a)(\exists m < b)\varphi(n,m)]
\end{align*}
where $\varphi$ is $\Sigma^0_n$ ($\Pi^0_n$).  The fragment $\mathrm{B}\Sigma^0_n$ ($\mathrm{B}\Pi^0_n$) is then $\mathrm{I}\Sigma^0_0$ plus the $\Sigma^0_n$ ($\Pi^0_n$) bounding scheme.

The following theorem summarizes the relationships among these fragments.

\begin{Theorem}[see \cite{Hajek:1998uh}~Theorem 2.4, \cite{Hajek:1998uh}~Theorem 2.5, and~\cite{Slaman:2004bh}]\label{thm-InductionSummary}
Let $n \in \omega$. Over $\pam$:
\begin{itemize}
\item $\mathrm{I}\Sigma^0_n$ and $\mathrm{I}\Pi^0_n$ are equivalent.
\item $\mathrm{B}\Sigma^0_{n+1}$ and $\mathrm{B}\Pi^0_n$ are equivalent.
\item $\mathrm{I}\Sigma^0_{n+1}$ is strictly stronger than $\mathrm{B}\Sigma^0_{n+1}$, which is strictly stronger than $\mathrm{I}\Sigma^0_n$.
\item If $n \geq 2$, then $\mathrm{I}\Delta^0_n$ and $\mathrm{B}\Sigma^0_n$ are equivalent (the proof uses the totality of the exponential function, which is provable in $\iso$).
\end{itemize}
\end{Theorem}

A \emph{cut} in a model $\Nb$ of $\pam$ is a set $I \subseteq \Nb$ such that $\forall n \forall m[(n \in I \andd m < n) \imp m \in I]$ and $\forall n(n \in I \imp n+1 \in I)$.  A cut $I \subseteq \Nb$ is called \emph{proper} if $I \neq \emptyset$ and $I \neq \Nb$.  Definable proper cuts witness failures of induction.  Suppose that $\Nb \models \pam$.  If the induction axiom for $\varphi$ fails in $\Nb$, then $\psi(n) = (\forall m < n)\varphi(m)$ defines a proper cut in $\Nb$, and if $\varphi$ defines a proper cut in $\Nb$, then the induction axiom for $\varphi$ fails in $\Nb$.

The following lemma, originally noticed by Friedman but by now part of the folklore, is key to many recursion-theoretic constructions in models with limited induction, including the main construction in this work.
\begin{Lemma}\label{lem-cut}
If $\Nb \models \bst + \neg\ist$, then there are a proper $\Sigma^0_2$ cut $I \subseteq \Nb$ and an increasing, cofinal function $c \colon I \imp \Nb$ whose graph is $\Delta^0_2$.
\end{Lemma}

\begin{proof}
Let $\varphi(n)$ be a $\Sigma^0_2$ formula witnessing the failure of $\ist$.  That is, $\varphi(0) \andd \forall n (\varphi(n) \imp \varphi(n+1)) \andd \exists n \neg\varphi(n)$.  Let $I = \{n : (\forall m < n) \varphi(m)\}$.  $I$ is a proper cut, and using $\bst$ one proves that $I$ is $\Sigma^0_2$.  Let $\theta$ be $\Pi^0_1$ such that $I = \{n : \exists m \theta(n,m)\}$.  Define the function $c$ by $c(n) = \mu m \theta(n,m)$ and observe that the graph of $c$ is $\Delta^0_2$.  By $\iso$, if there is an $m$ such that $\theta(n,m)$, then there is a least such $m$.  Therefore $\dom(c) = I$.  Furthermore, $\ran(c)$ is unbounded, for if $\exists b(\forall n \in I)(c(n) < b)$, then $\forall n(n \in I \biimp (\exists m < b)\theta(n,m))$, which constitutes a violation of $\iso$.  If necessary, using $\bst$ we can dominate $c$ by an increasing function with the same domain whose graph is still $\Delta^0_2$.
\end{proof}

\subsection{Fragments of second-order arithmetic}

Full second-order arithmetic consists of $\pam$ plus the universal closures of the induction axiom
\begin{align*}
[0 \in X \andd \forall n(n \in X \imp n+1 \in X)] \imp \forall n (n \in X)
\end{align*}
and the comprehension scheme
\begin{align*}
\exists X \forall n(n \in X \biimp \varphi(n)),
\end{align*}
where $\varphi$ is any formula in the language of second-order arithmetic in which $X$ is not free.  In the setting of second-order arithmetic, formulas may have free second-order parameters, and `universal closure' means closure under both first-order and second-order universal quantifiers.

Fragments of second-order arithmetic are obtained by replacing the induction axiom by an induction scheme as in the first-order case and by limiting the comprehension scheme to formulas of a certain complexity.  We emphasize again that in the second-order setting a formula may have free second-order parameters that are universally quantified in the corresponding induction axiom, hence an induction axiom holding in some second-order structure means that it holds relative to every second-order object in that structure.  When studying reverse mathematics, we also produce fragments of second-order arithmetic by adding the statement of a well-known theorem to another fragment, as is the case in the system weak K\"onig's lemma described below.  This work is concerned with the first two of the Big Five fragments of second-order arithmetic, \emph{recursive comprehension axiom} ($\rcasys$) and \emph{weak K\"onig's lemma} ($\wklsys$), as well as various fragments defined by statements asserting the existence of diagonally non-recursive functions.

$\rcasys$ is the fragment consisting of $\pam$, the second-order $\Sigma^0_1$ induction scheme (which we still refer to as $\iso$ in this setting), and the $\Delta^0_1$ comprehension scheme, which consists of the universal closures of all formulas of the form
\begin{align*}
\forall n (\varphi(n) \biimp \psi(n)) \imp \exists X \forall n(n \in X \biimp \varphi(n)),
\end{align*}
where $\varphi$ is $\Sigma^0_1$, $\psi$ is $\Pi^0_1$, and $X$ is not free in $\varphi$.

The equivalences and implications of Theorem~\ref{thm-InductionSummary} hold over $\rcasys$ in the second-order setting.  Most relevant to our purposes are that
\begin{itemize}
\item for all $n \in \omega$, $\rcasys + \mathrm{I}\Sigma^0_n$ and $\rcasys + \mathrm{I}\Pi^0_n$ are equivalent (in particular, $\rcasys \vdash \mathrm{I}\Pi^0_1$);
\item $\rcasys + \ist$ is strictly stronger than $\rcasys + \bst$, which is strictly stronger than $\rcasys$; and
\item $\rcasys + \bst \vdash \mathrm{I}\Delta^0_2$ (in particular, models of $\rcasys + \bst$ have no $\Delta^0_2$-definable cuts). 
\end{itemize}

An important aid to working in $\rcasys$ is the fact that $\rcasys$ proves the bounded $\Sigma^0_1$ comprehension scheme (see \cite{Simpson:2009vv}~Theorem~II.3.9), which consists of the universal closures of all formulas of the form
\begin{align*}
\forall n \exists X \forall i[i \in X \biimp (i < n \andd \varphi(i))],
\end{align*}
where $\varphi$ is a $\Sigma^0_1$ formula in which $X$ is not free.  Contrastingly, adding the full $\Sigma^0_1$ comprehension scheme to $\rcasys$ is equivalent to adding comprehension for all arithmetical formulas and results in a stronger system denoted $\acasys$ (see \cite{Simpson:2009vv}~Theorem~III.1.3).

$\rcasys$ proves sufficient number-theoretic facts to implement the codings of sequences of numbers as numbers that are typical in recursion theory.  See \cite{Simpson:2009vv}~Section~II.2 for a carefully formalized development of such a coding.  Thus in $\rcasys$ we can interpret the existence of the set $\Nb^{<\Nb}$ of all finite sequences (also called strings) and, more generally, give the usual definition of a \emph{tree} as subset of $\Nb^{<\Nb}$ that is closed under initial segments.  We now fix our notation and terminology concerning strings and trees.  Let $k,s \in \Nb$, $\sigma, \tau \in \Nb^{<\Nb}$, $f \colon \Nb \imp \Nb$ be a function, and $T \subseteq \Nb^{<\Nb}$ be a tree.  Then
\begin{itemize}
\item $k^{<\Nb}$ is the set of strings over $\{0,1,\dots,k-1\}$, $k^s$ is the set of strings in $k^{<\Nb}$ of length exactly $s$, and $k^{<s}$ is the set of strings in $k^{<\Nb}$ of length less than $s$;
\item $|\sigma|$ is the length of $\sigma$;
\item $\sigma \subseteq \tau$ means that $\sigma$ is a substring of $\tau$;
\item $f \restriction n$ is the string $\la f(0), f(1), \dots, f(n-1) \ra$;
\item $\sigma \subseteq f$ means that $\sigma$ is an initial segment of $f$ (i.e., $\sigma = f \restriction |\sigma|)$;
\item $f$ is a \emph{path through $T$} if $\forall n(f \restriction n \in T)$.
\end{itemize}

Weak K\"onig's lemma ($\wklstat$) is the statement ``every infinite subtree of $2^{<\Nb}$ has an infinite path,''  and $\wklsys$ is the fragment $\rcasys + \wklstat$.  $\wklsys$ captures compactness arguments, and $\wklstat$ is equivalent to many classical theorems over $\rcasys$.  For example, the equivalence of $\wklstat$ with the Heine-Borel compactness of $[0,1]$, the extreme value theorem, G\"odel's completeness theorem, and Brouwer's fixed point theorem can all be found in~\cite{Simpson:2009vv}.

Suppressing the basic relations and functions, a structure in the language of second-order arithmetic is officially a pair $(\Nb, \MP S)$, where the first-order part $\Nb$ is some set and the second-order part $\MP S$ is a collection of subsets of $\Nb$.  However, via the simple coding of pairs possible in $\rcasys$ and the identification of a function $f \colon \Nb \imp \Nb$ with its graph $\{\la n, m \ra : f(n)=m\}$, one immediately sees that it is equivalent to consider structures in which the second-order part is a collection of functions $f \colon \Nb \imp \Nb$.  Thus we use the functional variant of second-order structures because it is the more natural setting for our study.

\subsection{Turing reducibility and Turing functionals}

The standard definition of Turing reducibility in $\rcasys$ is \cite{Simpson:2009vv}~Definition~VII.1.4, which essentially says that $Y \leqT X$ if $Y$ is both r.e.\ and co-r.e.\ in $X$.

\begin{Definition}[\cite{Simpson:2009vv}~Definition~VII.1.4]
Fix a universal lightface $\Pi^0_1$ formula $\pi(e,m,X)$ with exactly the displayed variables free.  For $X, Y \subseteq \Nb$, we say that \emph{$Y$ Turing reduces to $X$} ($Y \leqT X$) if there are $e_0, e_1 \in \Nb$ such that, for all $m$, $m \in Y \biimp \pi(e_0, m, X)$ and $m \notin Y \biimp \pi(e_1,m,X)$.
\end{Definition}

Note that in the preceding definition $m\in Y \biimp \neg  \pi(e_1,m,X)$, so $\neg \pi(e_1,m,X)$ is a $\Sigma^0_1$ formula essentially witnessing that $Y$ is r.e.~in $X$.  Extending this notion, we can formalize statements involving recursive functionals as used in \cite{Soare:1987vq}~Section~III.1.  For example, given $e \in \Nb$ we write $\Phi_e^f(n) = m$ to represent a formula asserting that there is a coded sequence of configurations of the $e$\textsuperscript{th} Turing machine that starts with the machine's initial configuration for input $n$, ends with the machine's output configuration for output $m$, and is such that each configuration in the sequence follows from the previous one by the rules of the machine when equipped with oracle $f$.  In this way we think of $\Phi_e^f$ as a partial $f$-recursive function as usual, and for two functions $f,g \colon \Nb \imp \Nb$, $g \leqT f$ if and only if there is an $e$ such that $g = \Phi_e^f$.  We make the familiar definitions that $\Phi_e^f(n)\da$ if there is an $m$ such that $\Phi_e^f(n)=m$ and that $\Phi_e^f(n)\ua$ otherwise.  Similarly, $\Phi_{e,s}^f(n)\da$ if $\Phi_e^f(n)\da$ within $s$ computational steps and $\Phi_{e,s}^f(n)\ua$ otherwise.  We follow the usual convention that the number of steps in a computation relative to a partial oracle is bounded by the first position where the oracle is undefined, such as with computations of the form $\Phi_e^\sigma(n)$ and $\Phi_e^{f \oplus \sigma}(n)$, where $\sigma$ is some finite string. 

The following notion will be useful to verify $\bst$ when constructing models.

\begin{Definition}
  We say that \emph{$Y$ is low relative to $X$} if $\Phi^Y_e(e)\da$ is equivalent to a $\Delta^0_2(X)$ statement.
\end{Definition}

\begin{Lemma}[\cite{Chong:2013wx} Proposition~4.14]\label{lem-LowBST}
If $Y$ is low relative to $X$ and $\bst$ holds relative to $X$ then $\bst$ also holds relative to $Y$.
\end{Lemma}

\subsection{Diagonally non-recursive functions in the formal setting}

We now introduce the statements expressing the existence of diagonally non-recursive functions that are the main focus of this paper.

\begin{Definition}
Let $f$ and $g$ be functions $\Nb \imp \Nb$, and let $k \in \Nb$.
\begin{itemize}
\item The function $g$ is \emph{$k$-bounded} if $\ran(g) \subseteq \{0,1,\dots,k-1\}$.
\item The function $g$ is \emph{diagonally non-recursive relative to $f$} ($g$ is $\dnrstat(f)$ for short) if $\forall e(g(e) \neq \Phi_e^f(e))$.
\item The function $g$ is \emph{$k$-bounded diagonally non-recursive relative to $f$} ($g$ is $\dnrstat(k,f)$ for short) if $g$ is $k$-bounded and $\dnrstat(f)$.
\end{itemize}
\end{Definition}

In a slight overloading of notation we also let $\dnrstat(f)$ denote the formal statement ``there is a $g$ that is $\dnrstat(f)$'' and let $\dnrstat(k,f)$ denote the formal statement ``there is a $g$ that is $\dnrstat(k,f)$.''

It is well-known that $\wklstat$ and $\forall f \dnrstat(k,f)$ are equivalent over $\rcasys$ for every fixed $k \in \omega$ with $k \geq 2$.  $\wklstat$ and $\forall f \dnrstat(2,f)$ are equivalent by the classic work of Jockusch and Soare~\cite{JockuschJr:1972wma}, and $\forall f \dnrstat(2,f)$ and $\forall f \dnrstat(k,f)$ are equivalent because if $k \in \omega$ and $k \geq 2$, then the proof of \cite{JockuschJr:1989vs}~Theorem~5 can be unwound in $\rcasys$.  It is also well-known that $\forall f \dnrstat(f)$ is strictly weaker than $\wklstat$ over $\rcasys$.  In fact, $\forall f \dnrstat(f)$ is strictly weaker than $\wwklstat$~\cite{AmbosSpies:2004cv}, which is strictly weaker than $\wklstat$~\cite{Yu:1990gl}.  The purpose of this work is to analyze the strengths of the statements $\exists k \forall f \dnrstat(k,f)$ and $\forall f \exists k \dnrstat(k,f)$ over $\rcasys$.  With a little care, it is possible to implement the proof of \cite{JockuschJr:1989vs}~Theorem~5 in $\rcasys + \ist$.  Hence the statements $\wklstat$, $\exists k \forall f \dnrstat(k,f)$, and $\forall f \exists k \dnrstat(k,f)$ are all equivalent over $\rcasys + \ist$.

\begin{Theorem}\label{thm-DNR2inISigma2}
$\rcasys + \ist + \forall f \exists k \dnrstat(k,f) \vdash \forall f \dnrstat(2,f)$.
\end{Theorem}

\begin{proof}
Suppose $g$ is $\dnrstat(2^k,f)$, and think of $2^k$ as the set of strings over $\{0,1\}$ of length $k$.  Define a partial $f$-computable function $b$ by
\begin{align*}
b(n) =
\begin{cases}
0 & \text{if $\Phi_n^f(n) = 0$}\\
1 & \text{if $\Phi_n^f(n) > 0$}\\
\uparrow & \text{if $\Phi_n^f(n)\ua$},
\end{cases}
\end{align*}
let $\ell \colon \Nb^k \imp \Nb$ be a partial computable function such that
\begin{align*}
(\forall \vec n \in \Nb^k)(\forall x \in \Nb)(\Phi_{\ell(\vec n)}^f(x) = \la b(n_0), b(n_1), \dots, b(n_{k-1})\ra),
\end{align*}
and let $h \colon k \times \Nb^k \imp 2$ be the partial $g$-computable function defined by the equation
\begin{align*}
g(\ell(\vec n)) = \la h(0,\vec n), h(1,\vec n), \dots, h(k-1,\vec n) \ra.
\end{align*}

By the $\Pi^0_2$ least element principle, a consequence of $\rcasys + \ist$, let $i$ be least such that
\begin{align*}
(\forall \vec n \in \Nb^i)(\exists \vec m \in \Nb^{k-i})(\forall j < k-i)(h(i+j,\vec{n}^\smf\vec{m}) = \Phi_{m_j}^f(m_j)).
\end{align*}
Notice that $i > 0$, for otherwise we would have an $\vec m \in \Nb^k$ such that $(\forall j < k)(h(j,\vec m) = \Phi_{m_j}^f(m_j) = b(m_j))$, in which case $g(\ell(\vec m)) = \Phi_{\ell(\vec m)}^f(\ell(\vec m))$, contradicting that $g$ is $\dnrstat(2^k,f)$.
Fix $\vec n \in \Nb^{i-1}$ such that $(\forall \vec m \in \Nb^{k-i+1})(\exists j < k-i+1)(h(i+j,\vec{n}^\smf\vec{m}) \neq \Phi_{m_j}^f(m_j))$.  We can now describe a $\dnrstat(2,f)$ function that is $\leqT f \oplus g$.  Given $x \in \Nb$, search for an $\vec m \in \Nb^{k-i}$ such that $(\forall j < k-i)(h(i+j,\vec{n}^\smf x^\smf\vec{m}) = \Phi_{m_j}^f(m_j))$, then output $h(i-1,\vec{n}^\smf x^\smf\vec{m})$.  Such an $\vec m$ exists by choice of $i$, and $h(i-1,\vec{n}^\smf x^\smf\vec{m}) \neq \Phi_x^f(x)$ by choice of $\vec n$.
\end{proof}

In Section~\ref{sec-Main}, we show that $\rcasys + \bst$ does not suffice to prove the equivalences of $\wklstat$, $\exists k \forall f \dnrstat(k,f)$, and $\forall f \exists k \dnrstat(k,f)$.  Specifically, we prove
\begin{itemize}
\item Theorem~\ref{thm-noDNRK}: $\rcasys + \bst + \forall f \exists k \dnrstat(k,f) \nvdash \exists k \forall f \dnrstat(k,f)$, and

\item Theorem~\ref{thm-noWKL}: $\rcasys + \bst + \exists k \forall f \dnrstat(k,f) \nvdash \wklstat$.
\end{itemize}
Hence, over $\rcasys + \bst$, $\forall f \exists k \dnrstat(k,f)$ is strictly weaker than $\exists k \forall f \dnrstat(k,f)$, which is strictly weaker than $\wklstat$.  These results are, in a sense, as strong as possible.  It is of course natural to ask if there is a reversal of Theorem~\ref{thm-DNR2inISigma2}.  That is, it is natural to ask if $\rcasys \vdash (\exists k \forall f \dnrstat(k,f) \imp \wklstat) \imp \ist$.  However, this is readily seen not to be the case because $\wklsys \nvdash \ist$.  In fact, $\wklsys \nvdash \bst$.  No reversal over $\rcasys + \bst$ is possible either.  That is, $\rcasys + \bst \nvdash (\exists k \forall f \dnrstat(k,f) \imp \wklstat) \imp \ist$.  This is because $\wklsys + \bst \nvdash \ist$.  These comments all follow from the facts that $\wklsys$ is $\Pi^1_1$-conservative over $\rcasys$ (see~\cite{Simpson:2009vv}~Corollary~IX.2.6) and that $\wklsys + \bst$ is $\Pi^1_1$-conservative over $\rcasys + \bst$ (see~\cite{Hajek:1993vb} or adapt the proof of \cite{Simpson:2009vv}~Corollary IX.2.6).

\section{A little combinatorics of trees}

In this short section we isolate two facts concerning the combinatorics of finite trees.  These facts appear in~\cite{AmbosSpies:2004cv}, but we repeat them here for the sake of completeness and because it is important for our purposes to emphasize that the proofs are formalizable in the first-order fragment $\iso$ and hence in $\rcasys$.

\begin{Definition}[see \cite{AmbosSpies:2004cv}~Definition~2.3]{\ }
\begin{itemize}
\item The \emph{trunk} of a finite tree $T \subseteq \Nb^{<\Nb}$ is the longest $\sigma \in T$ such that every element of $T$ is comparable with $\sigma$.
\item A finite tree $T \subseteq \Nb^{<\Nb}$ with trunk $\sigma$ is \emph{$\geq\! n$-branching} if every $\tau \supseteq \sigma$ in $T$ that is not a leaf has at least $n$ immediate successors.
\end{itemize}
\end{Definition}

\begin{Lemma}[$\iso$; see \cite{AmbosSpies:2004cv}~Lemma~2.5]\label{lem-CoverWith2Trees}
Let $m \geq 1$, let $T \subseteq \Nb^{<\Nb}$ be a finite, $\geq\! 2m$-branching tree with trunk $\sigma$, and let $P_0$ and $P_1$ be finite trees such that $T \subseteq P_0 \cup P_1$.  Then there is a $\geq\! m$-branching tree $S \subseteq T$ with trunk $\sigma$ such that $\leaves(S) \subseteq \leaves(T)$ and either $S \subseteq P_0$ or $S \subseteq P_1$.
\end{Lemma}
\begin{proof}
For the purposes of this proof, define $\depth(T,\sigma) = \max\{|\tau|-|\sigma| : \tau \in T\}$ for a finite tree $T \subseteq \Nb^{<\Nb}$ with trunk $\sigma$.  We prove the lemma by induction on $\depth(T,\sigma)$.  If $\depth(T,\sigma) = 0$, then $T = \sigma$ (we identify $\sigma$ with $\{\tau : \tau \subseteq \sigma\}$ for simplicity).  Thus $T \subseteq P_0 \cup P_1$ implies that $\sigma \in P_i$ for some $i < 2$, which implies that $T \subseteq P_i$.  Now suppose that $\depth(T,\sigma) = n+1$.  Let $(\tau_j : j < 2m)$ be the first $2m$ immediate successors of $\sigma$ in $T$, and for each $j < 2m$, let $T_j = \{\tau \in T : \tau \supseteq \tau_j\}$.  For each $j < 2m$, $T_j$ is a $\geq\! 2m$-branching tree with trunk $\tau_j$, $\depth(T_j,\tau_j) \leq n$, and $T_j \subseteq P_0 \cup P_1$.  By induction, for each $j < 2m$ there are an $i_j < 2$ and a $\geq\! m$-branching subtree $S_j \subseteq T_j$ with trunk $\tau_j$ such that $\leaves(S_j) \subseteq \leaves(T_j)$ and $S_j \subseteq P_{i_j}$.  There is then an $i < 2$ such that $i_j = i$ for at least $m$ of the $i_j$.  Let $S = \bigcup \{S_j : j < 2m \andd i_j = i\}$.  Then $S$ is a desired $\geq\! m$-branching subtree of $T$ with trunk $\sigma$ such that $\leaves(S) \subseteq \leaves(T)$ and $S \subseteq P_i$.
\end{proof}

\begin{Lemma}[$\iso$; see \cite{AmbosSpies:2004cv}~Lemma~2.6]\label{lem-CoverWithNTrees}
Let $m, n \geq 1$, let $T$ be a finite, $\geq\! m2^{n-1}$-branching tree with trunk $\sigma$, and let $(P_i : i < n)$ be finite trees such that $T \subseteq \bigcup_{i<n}P_i$.  Then there are an $i < n$ and a $\geq\! m$-branching tree $S \subseteq T$ with trunk $\sigma$ such that $\leaves(S) \subseteq \leaves(T)$ and $S \subseteq P_i$. 
\end{Lemma}
\begin{proof}
By induction on $n$.  The case $n=1$ is trivial.  Suppose that $T$ is a finite, $\geq\! m2^n$-branching tree with trunk $\sigma$ such that $T \subseteq \bigcup_{i<n+1}P_i$.  By Lemma~\ref{lem-CoverWith2Trees}, there is a $\geq\! m2^{n-1}$-branching tree $S \subseteq T$ with trunk $\sigma$ such that $\leaves(S) \subseteq \leaves(T)$ and either $S \subseteq \bigcup_{i<n}P_i$ or $S \subseteq P_n$.  If $S \subseteq P_n$ we are done.  If $S \subseteq \bigcup_{i<n}P_i$, then by induction there are an $i < n$ and a $\geq\! m$-branching tree $S_0 \subseteq S$ with trunk $\sigma$ such that $\leaves(S_0) \subseteq \leaves(S) \subseteq \leaves(T)$ and $S_0 \subseteq P_i$ as desired. 
\end{proof}

\section{Low $\dnrstat(k,f)$ functions that avoid $\dnrstat(b,h)$ functions}\label{sec-Main}

Consider a countable model $M = (\Nb, \MP S)$ of $\rcasys + \bst$ with a proper $\Sigma^0_2$ cut.  Let $f \in \MP S$, $n \in \omega$, $\vec h$ an $n$-tuple of elements of $\MP S$, and $\vec b$ an $n$-tuple of elements of $\Nb$ be such that $(\forall i < n)(h_i \leqT f)$ and $(\forall i < n)(\text{$f$ computes no $\dnrstat(b_i, h_i)$ function})$.  Our goal is to produce a function $g$ (outside of $\MP S$) that is $\dnrstat(k,f)$ for some $k \in \Nb$ but is such that that $f \oplus g$ computes no $\dnrstat(b_i, h_i)$ function for any $i < n$.

In $\rcasys$, define the function $K(b,s)$ by $K(b,0) = 2$ and $K(b,s+1) = K(b,s)2^{s^2+b+1}$.  Our main technical result is the following theorem.

\begin{Theorem}\label{thm-DNRKwithoutDNR2}
Let 
\begin{itemize}
\item $M = (\Nb, \MP S)$ be a countable model of $\rcasys + \bst$ with a proper $\Sigma^0_2$ cut $I$;
\item $n \in \omega$, $f \in \MP S$, $\vec h$ an $n$-tuple of elements of $\MP S$, and $\vec b$ an $n$-tuple of elements of $\Nb$ be such that
\begin{itemize}
\item $(\forall i < n)(h_i \leqT f)$ and
\item $(\forall i < n)(\text{$f$ computes no $\dnrstat(b_i, h_i)$ function})$;
\end{itemize}
\item $b_{\max} = \max \vec b$;
\item $k_0 \in \Nb$ be such that $(\forall i \in I)(k_0 > i)$;
\item $k = K(b_{\max},k_0)$.
\end{itemize}
Then there is a $\dnrstat(k,f)$ function $g$ such that $f \oplus g$ is low relative to $f$ and such that $f \oplus g$ computes no $\dnrstat(b_i, h_i)$ function for any $i < n$.
\end{Theorem}

The conclusion that $f \oplus g$ is low relative to $f$ in Theorem~\ref{thm-DNRKwithoutDNR2} ensures that $f \oplus g$ preserves $\bst$.

Before we continue with the proof of Theorem~\ref{thm-DNRKwithoutDNR2}, we point out that its simplest case provides an interesting example concerning recursion theory in models with limited induction.

\begin{Corollary}
If $\Nb$ satisfies $\bst$ but not $\ist$, then there is a $k \in \Nb$ and a $k$-bounded diagonally non-recursive function that computes no $2$-bounded diagonally non-recursive function.
\end{Corollary}

In their proof of \cite{AmbosSpies:2004cv}~Theorem~2.1, Ambos-Spies et al.\ construct a diagonally non-recursive function $g \colon \omega \imp \omega$ (with necessarily unbounded range) that computes no $2$-bounded diagonally non-recursive function.  In fact, given a recursive $h$, they construct a diagonally non-recursive $g$ that computes no $h$-bounded diagonally non-recursive function.  Our proof of Theorem~\ref{thm-DNRKwithoutDNR2} is essentially the proof of \cite{AmbosSpies:2004cv}~Theorem~2.1 implemented inside of a $\Sigma^0_2$ cut as provided by Lemma~\ref{lem-cut}.  With this strategy, the Ambos-Spies et al.\ construction is completed in a bounded number of steps, thereby producing a diagonally non-recursive function $g$ with bounded range that does not compute a $2$-bounded diagonally non-recursive function.

We build a function $g$ satisfying the conclusion of Theorem~\ref{thm-DNRKwithoutDNR2} in a sequence of finite extensions.  Throughout the construction, we maintain a coded finite set $D$ of divergent computations according to the following definition.  Fix $k$ as in the statement of Theorem~\ref{thm-DNRKwithoutDNR2}.  Henceforth and through the proof of Theorem~\ref{thm-DNRKwithoutDNR2}, all strings are elements of $k^{<\Nb}$ and all trees are subtrees of $k^{<\Nb}$.

\begin{Definition}
Let $f \colon \Nb \imp \Nb$ be a function.
\begin{itemize}
\item A string \emph{$\sigma$ admits $\geq\! m$-branching $f$-convergence for $\la e, x \ra$} if there is a finite $\geq\! m$-branching tree $T$ with trunk $\sigma$ such that $(\forall \alpha \in \leaves(T))(\Phi_e^{f \oplus \alpha}(x)\da)$.

\item A string \emph{$\sigma$ forces $\geq\! m$-branching $f$-divergence for $\la e, x \ra$} if $\sigma$ does not admit $\geq\! m$-branching $f$-convergence for $\la e, x \ra$.

\item Let $D$ be a finite coded subset of $\Nb$.  A string $\sigma$ \emph{forces $\geq\! m$-branching $f$-divergence for $D$} if $\sigma$ forces $\geq\! m$-branching $f$-divergence for every $\la e, x \ra \in D$.
\end{itemize}
\end{Definition}

The following lemma is essentially Lemma~2.8 in~\cite{AmbosSpies:2004cv}.

\begin{Lemma}[$\rcasys$; see \cite{AmbosSpies:2004cv}~Lemma~2.8]\label{lem-PreserveDivergence}
Suppose $\sigma$ is a string and $D$ is a finite coded set such that $\sigma$ forces $\geq\! m$-branching $f$-divergence for $D$, where $m2^{|D|} < k$.  Then every $\geq\! m2^{|D|}$-branching tree with trunk $\sigma$ has a leaf that forces $\geq\! m$-branching $f$-divergence for $D$.
\end{Lemma}

\begin{proof}
Suppose $\sigma$ forces $\geq\! m$-branching $f$-divergence for $D$, suppose $T$ is a finite $\geq\! m2^{|D|}$-branching tree with trunk $\sigma$, and suppose for a contradiction that no leaf of $T$ forces $\geq\! m$-branching $f$-divergence for $D$.  Enumerate $D$ as  $D = \{\la e_i, x_i \ra : i < |D|\}$, and, using bounded $\Sigma^0_1$ comprehension, define a function $j \colon \leaves(T) \imp |D|$ by letting $j(\alpha)$ be least such that $\alpha$ admits $\geq\! m$-branching $f$-convergence for $\la e_{j(\alpha)}, x_{j(\alpha)} \ra$.  For each $i < |D|$, let $P_i$ be the tree consisting of the strings in $T$ extendible to an $\alpha \in \leaves(T)$ with $j(\alpha) = i$.  Then $T \subseteq \bigcup_{i<|D|}P_i$, so by Lemma~\ref{lem-CoverWithNTrees} there is a tree $T' \subseteq T$ that has trunk $\sigma$, is $\geq\! m$-branching, and is contained $P_i$ for some $i < |D|$.  For each $\alpha \in \leaves(T')$, let $T_\alpha$ be a $\geq\! m$-branching tree with trunk $\alpha$ such that $(\forall \beta \in \leaves(T_\alpha))(\Phi_{e_i}^{f \oplus \beta}(x_i)\da)$.  Then $T' \cup \bigcup_{\alpha \in \leaves(T')}T_\alpha$ witnesses that $\sigma$ admits $\geq\! m$-branching $f$-convergence for $\la e_i, x_i \ra$, a contradiction.
\end{proof}

The construction proceeds in stages.  In stages $s \equiv 0 \mod n+2$, we satisfy requirements ensuring that $g$ is total.  In stages $s \equiv i+1 \mod n+2$ for $i < n$, we satisfy blocks of requirements ensuring that $f \oplus g$ computes no $\dnrstat(b_i,h_i)$ function.  In stages $s \equiv n+1 \mod n+2$, we satisfy blocks of requirements ensuring that $f \oplus g$ is low relative to $f$.  In the end, $g$ satisfies $\ran(g) \subseteq k$ because we only consider extensions by strings $\sigma \in k^{<\Nb}$, and $g$ is diagonally non-recursive relative to $f$ because we ensure the divergence of $\Phi_{e_0}^{f \oplus g}(e_0)$, where $e_0$ is an index such that $(\Phi_{e_0}^{f \oplus g}(e_0)\da) \biimp \exists e(g(e) = \Phi_e^f(e))$.  To make the non-$\dnrstat(b_i,h_i)$ requirements more manageable, we condense a block of non-$\dnrstat(b_i,h_i)$ requirements into a single requirement.

\begin{Definition}
Let $b \in \Nb$, let $f \colon \Nb \imp b$, and let $h \colon \Nb \imp \Nb$.  We say that $f$ is \emph{eventually $\dnrstat(b,h)$} if $\exists n(\forall e > n)(f(e) \neq \Phi_e^h(e))$.
\end{Definition}

Define a primitive recursive function $d \colon \Nb^3 \imp \Nb$ such that, given functions $h \leqT f$, an index $e$ such that $\forall x(\Phi_e^f(x) = \Phi^h_x(x))$, and bounds $a$ and $b$, $d(e,a,b)$ is an index for a program such that, for every $\ell \in \Nb$ and function $g$, $\Phi_{d(e,a,b)}^{f \oplus g}(\ell)$ searches for a pair $\la i,s \ra$ such that $i<a$, $\Phi_{i,s}^{f \oplus g}(\ell) < b$, and $\neg(\exists \ell_0<\ell)(\Phi_{i,\ell}^{f \oplus g}(\ell_0) = \Phi_{\ell_0,\ell}^h(\ell_0))$.  If such a pair is found, then $\Phi_{d(e,a,b)}^{f \oplus g}(\ell) = \Phi_{i,s}^{f \oplus g}(\ell)$ for the first such pair.  Otherwise, $\Phi_{d(e,a,b)}^{f \oplus g}(\ell)\ua$.

\begin{Lemma}[$\rcasys$]\label{lem-DNRblock}
For any functions $f$, $g$, $h$ and $e, a, b \in \Nb$ as above, if there is an $i < a$ such that $\Phi_i^{f \oplus g}$ is $\dnrstat(b,h)$, then $\Phi_{d(e,a,b)}^{f \oplus g}$ is eventually $\dnrstat(b,h)$.
\end{Lemma}

\begin{proof}
The fact that $\Phi_i^{f \oplus g}$ is $\dnrstat(b,h)$ for some $i<a$ ensures that $\Phi_{d(e,a,b)}^{f \oplus g}$ is total.  The finite set $X = \{j < a : \exists \ell,s (\Phi_{j,s}^{f \oplus g}(\ell) = \Phi_{\ell,s}^h(\ell))\}$ exists by bounded $\Sigma^0_1$ comprehension.  $\bso$ then provides a bound $N$ such that $(\forall j \in X)(\exists \ell,s < N)(\Phi_{j,s}^{f \oplus g}(\ell) = \Phi_{\ell,s}^h(\ell))$.  To show that $\Phi_{d(e,a,b)}^{f \oplus g}$ is eventually $\dnrstat(b,h)$, we show that $(\forall \ell > N)(\Phi_{d(e,a,b)}^{f \oplus g}(\ell) \neq \Phi_\ell^h(\ell))$.  Suppose for a contradiction that $\ell > N$ and that $\Phi_{d(e,a,b)}^{f \oplus g}(\ell) = \Phi_\ell^h(\ell)$.  By the definition of $d(e,a,b)$, there is a $j<a$ such that $\Phi_{d(e,a,b)}^{f \oplus g}(\ell) = \Phi_j^{f \oplus g}(\ell)$ and $\neg(\exists \ell_0<\ell)(\Phi_{j,\ell}^{f \oplus g}(\ell_0) = \Phi_{\ell_0,\ell}^h(\ell_0))$.  The equation $\Phi_j^{f \oplus g}(\ell) = \Phi_{d(e,a,b)}^{f \oplus g}(\ell) = \Phi_\ell^h(\ell)$ implies that $j \in X$ and hence that there are $\ell_0, s < N$ such that $\Phi_{j,s}^{f \oplus g}(\ell_0) = \Phi_{\ell_0,s}^h(\ell_0)$.  Now, $N < \ell$ and therefore $(\exists \ell_0 < \ell)(\Phi_{j,\ell}^{f \oplus g}(\ell_0) = \Phi_{\ell_0,\ell}^h(\ell_0))$, a contradiction.
\end{proof}

Lemma~\ref{lem-AvoidDNR2} and Lemma~\ref{lem-Low} below aid the construction of a function witnessing Theorem~\ref{thm-DNRKwithoutDNR2}.  Their proofs are straightforward if one assumes $\ist$.  The more complicated arguments below are necessary for our purposes because they are compatible with $\bst + \neg\ist$.

\begin{Lemma}[$\rcasys + \bst$]\label{lem-AvoidDNR2}
Let $f$ and $h$ be functions and let $b \in \Nb$ be such that $f$ computes no $\dnrstat(b,h)$ function.  Let $\sigma$ be a string, let $D$ be a finite coded set, and let $m \in \Nb$ be such that $\sigma$ forces $\geq\! m$-branching $f$-divergence for $D$, where $m2^{|D|+b} < k$.  Then for every finite coded set $E$ there are a string $\sigma' \supseteq \sigma$ and a finite coded set $D' \supseteq D$ such that
\begin{itemize}
\item[(i)] $|D'| \leq |D|+|E|$,
\item[(ii)] $\sigma'$ forces $\geq\! m2^{|D|+b}$-branching $f$-divergence for $D'$, and
\item[(iii)] for each $e \in E$, $(\exists \ell > |\sigma|)(\Phi_e^{f \oplus \sigma'}(\ell) \geq b)$, $(\exists \ell > |\sigma|)(\Phi_e^{f \oplus \sigma'}(\ell) = \Phi_\ell^h(\ell))$, or $\exists \ell(\la e, \ell \ra \in D')$.
\end{itemize}
\end{Lemma}

\begin{proof}
We prove the lemma in $\wklsys + \bst$, which suffices because $\wklsys + \bst$ is $\Pi^1_1$-conservative over $\rcasys + \bst$ (see~\cite{Hajek:1993vb} or adapt the proof of \cite{Simpson:2009vv}~Corollary IX.2.6).  We thus construct an infinite tree $T$ with trunk $\sigma$ such that every infinite path through $T$ has an initial segment $\sigma'$ and a corresponding set $D'$ that satisfy the conclusion of the lemma.

The tree $T$ grows in stages $(T_s : s \in \Nb)$.  In order to describe the growth of $T$, we represent $T_s$ as $T_s = \bigcup_{\tau \in R_s} T_s(\tau)$, where $R_s \subseteq k^{<\Nb}$ is finite and, for each $\tau \in R_s$, $T_s(\tau)$ is the tree $\tau^\smf k^{<t}$ for some $t \in \Nb$.  Notice that $T_s(\tau)$ has trunk $\tau$.  As the construction proceeds, the component trees $T_s(\tau)$ are either \emph{alive}, in which case they are extended, or \emph{dead}, in which case they are not extended.  If $T_s(\tau)$ is dead, then no string that is a proper extension of a leaf of $T_s(\tau)$ is ever added to $T$.  During the course of the construction, a component tree $T_s(\tau)$ may be rewritten as a union of new component trees $\bigcup_{\eta \in \leaves(T_s(\tau))}T_{s+1}(\eta)$, where $T_{s+1}(\eta) = \eta$ for each $\eta \in \leaves(T_s(\tau))$, to allow the branches of $T_s(\tau)$ to grow according to different criteria.  In this situation, when we update $R_s$ to $R_{s+1}$, we remove $\tau$ and add the elements of $\leaves(T_s(\tau))$.  To each $\tau \in R_s$ we also associate a finite set $M(\tau)$ of requirements that have been met.

At stage $0$, let $R_0 = \{\sigma\}$, $T_0(\sigma) = \sigma$, and $M(\sigma) = \emptyset$.  $T_0(\sigma)$ is alive at stage $0$.

At the beginning of stage $s+1$, we have $T_s$ represented as $T_s = \bigcup_{\tau \in R_s} T_s(\tau)$, and we have the corresponding sequence of met requirements $(M(\tau) : \tau \in R_s)$.  For each $\tau \in R_s$ do the following:

\begin{itemize}
\item[(a)] If $T_s(\tau)$ is dead, put $\tau$ in $R_{s+1}$ and let $T_{s+1}(\tau) = T_s(\tau)$.  $T_{s+1}(\tau)$ is dead.

\item[(b)] If $T_s(\tau)$ is alive:
\begin{itemize}
\item[(i)] If there are a $t \leq s$, a $\tau' \in R_t$ with $\tau' \subseteq \tau$, an $\la e, x \ra \in D$, and a $\geq\! m$-branching tree $S \subseteq k^{< s}$ with trunk $\tau'$ such that $(\forall \alpha \in \leaves(S))(\Phi_e^{f \oplus \alpha}(x)\da)$ (i.e., we learn at stage $s$ that $\tau'$ admits $\geq\! m$-branching $f$-convergence for some $\la e, x \ra \in D$), then put $\tau$ in $R_{s+1}$ and let $T_{s+1}(\tau) = T_s(\tau)$.  $T_{s+1}(\tau)$ is dead.

\item[(ii)] If (i) fails and there are an $e \in E \setminus M(\tau)$, an $x$ with $|\sigma| < x \leq s$, and a $\geq\! m2^{|D|}$-branching tree $S \subseteq T_s(\tau)$ with trunk $\tau$ such that either $(\forall \alpha \in \leaves(S))(\Phi_e^{f \oplus \alpha}(x) = \Phi_{x,s}^h(x))$ or $(\forall \alpha \in \leaves(S))(\Phi_e^{f \oplus \alpha}(x) \geq b)$, then choose the least such $e$, the least such $x$ for the chosen $e$, and the least such $S$ for the chosen $e$ and $x$.  Put all leaves of $T_s(\tau)$ in $R_{s+1}$.  Let $T_{s+1}(\eta) = \eta$ and $M(\eta) = M(\tau) \cup \{e\}$ for all $\eta \in \leaves(T_s(\tau))$.  If $\eta \in \leaves(T_s(\tau))$ extends a leaf of $S$, then $T_{s+1}(\eta)$ is alive; otherwise $T_{s+1}(\eta)$ is dead.

\item[(iii)] If (i) and (ii) fail, then put $\tau$ in $R_{s+1}$ and let $T_{s+1}(\tau) = T_s(\tau)$.  $T_{s+1}(\tau)$ is alive.
\end{itemize}
\end{itemize}
Finally, for each $\tau \in R_{s+1}$ with $T_{s+1}(\tau)$ alive, grow $T_{s+1}(\tau)$ by extending each $\alpha \in \leaves(T_{s+1}(\tau))$ to $\alpha^\smf n$ for every $n < k$.  This concludes stage $s+1$.

The tree $T$ is an $(f \oplus h)$-recursive subtree of $k^{< \Nb}$ because every $\alpha \in k^{< \Nb}$ is either in $T_s(\tau)$ for some $\tau$ in some $R_s$ or properly extends some leaf of $T_s(\tau)$ for some $\tau$ in some $R_s$ where $T_s(\tau)$ is dead.

\begin{Claim}
$T$ is infinite.
\end{Claim}

\begin{proof}[Proof of Claim]
We prove that at the end of every stage $s$ there is some $\tau \in R_s$ such that $T_s(\tau)$ is alive.  $T_s(\tau)$ thus grows at the end of stage $s$, and therefore new strings are added to $T$ at every stage.  Hence $T$ is infinite.

For each $s$, let $Q_s$ be the tree of strings extendible to a $\tau \in R_s$ such that either $T_s(\tau)$ is alive or $T_s(\tau)$ died by item~(b) part~(i) at some stage $\leq s$.  $\iso$ suffices to prove that each $Q_s$ is $\geq\! m2^{|D|}$-branching with trunk $\sigma$.  This is because a $\tau \in \leaves(Q_s)$ is extended in $Q_{s+1}$ only by the result of acting according to item~(b) part~(ii) for $\tau$ at stage $s+1$, in which case the subtree of $T_s(\tau)$ consisting of strings extendible to an $\eta \in \leaves(T_s(\tau))$ with $T_{s+1}(\eta)$ alive is $\geq\! m2^{|D|}$-branching with trunk $\tau$.  Thus in $Q_{s+1}$, $\tau$ is appended by a $\geq\! m2^{|D|}$-branching tree.

Suppose for a contradiction that, at some stage $s$, $T_s(\tau)$ is dead for all $\tau \in R_s$.  Thus each $T_s(\tau)$ for $\tau \in \leaves(Q_s)$ died by item~(b) part~(i) at some stage $\leq s$.  For each $\tau \in \leaves(Q_s)$, let $\tau' \subseteq \tau$ be such that $\tau'$ admits $\geq\! m$-branching $f$-convergence for some $\la e, x \ra \in D$ as in item~(b) part~(i) at the time of $T_s(\tau)$'s death.  Let $R = \{\tau' : (\tau \in \leaves(Q_s)) \andd \neg(\exists \eta \in \leaves(Q_s))(\eta' \subset \tau')\}$.  Let $S$ be the tree of strings extendible to some $\tau' \in R$.  $S$ is $\geq\! m2^{|D|}$-branching with trunk $\sigma$, but no leaf of $S$ forces $\geq\! m$-branching $f$-divergence for $D$.  This contradicts Lemma~\ref{lem-PreserveDivergence}.
\end{proof}

By $\wklsys$, let $p$ be an infinite path through $T$.  Using bounded $\Sigma^0_1$ comprehension, let $s \in \Nb$ and $\sigma' \subset p$ maximize $|M(\sigma')|$ over all $s \in \Nb$ and $\sigma' \in R_s$ with $\sigma' \subset p$.  Observe that the construction never acts on $T_t(\sigma')$ according to item~(b) at any stage $t > s$.  If the construction acts at stage $t > s$ according to item~(b) part~(i), then $T_t(\sigma')$ dies and $p$ could not be a path through $T$.  If the construction acts at stage $t+1 > s$ according to item~(b) part~(ii), then $p$ must extend some $\eta \in \leaves(T_t(\sigma'))$ with $T_{t+1}(\eta)$ alive, and $|M(\eta)| > |M(\sigma')|$ for all such $\eta$.  This contradicts the choice of $\sigma'$ and $s$.  It follows that $\sigma' \in R_t$ for all stages $t \geq s$.

To find $D'$, we define a $\ell_{e,t} \in \Nb$ for every $e \in E \setminus M(\sigma')$ and every $t > s$ as follows.  Let $\ell_{e,t}$ be least $>|\sigma|$ such that no tree $S \subseteq T_t(\sigma')$ with trunk $\sigma'$ witnesses that $\sigma'$ admits $\geq\! m2^{|D|+b}$-branching $f$-convergence for $\la e, \ell_{e,t} \ra$.

\begin{Claim}
$(\forall e \in E \setminus M(\sigma'))(\exists t > s)(\forall t' > t)(\ell_{e,t'} = \ell_{e,t})$.
\end{Claim}

\begin{proof}[Proof of Claim]
Let $e \in E \setminus M(\sigma')$.  The numbers $\ell_{e,t}$ are increasing in $t$, so if $(\exists t)(\forall t' > t)(\ell_{e,t'} = \ell_{e,t})$ fails, then it must be that $\lim_{t \imp \infty} \ell_{e,t} = \infty$.  Thus suppose for a contradiction that $\lim_{t \imp \infty} \ell_{e,t} = \infty$.  We then compute an eventually $\dnrstat(b,h)$ function from $f$, contradicting the hypothesis that $f$ computes no $\dnrstat(b,h)$ function and hence no eventually $\dnrstat(b,h)$ function.

Given $x \in \Nb$, if $x \leq |\sigma|$ then output $0$.  If $x > |\sigma|$, run the construction to a stage $t > s$ such that $t, \ell_{e,t} > x$ and there are an $i < b$ and a $\geq\! m2^{|D|}$-branching tree $S \subseteq T_t(\sigma')$ such that $(\forall \alpha \in \leaves(S))(\Phi_e^{f \oplus \alpha}(x) = i)$.  Then output the least such $i$.  This procedure describes a $b$-valued partial $f$-recursive function $\Phi^f$.  To see that $\Phi^f(x)$ converges for $x > |\sigma|$, observe that there is a $t > s,x$ such that $\ell_{e,t} > x$ because $\lim_{t \imp \infty} \ell_{e,t} = \infty$ and that at such a stage $t$, by the definition of $\ell_{e,t}$, there must be a tree $S' \subseteq T_t(\sigma')$ with trunk $\sigma'$ witnessing that $\sigma'$ admits $\geq\! m2^{|D|+b}$-branching $f$-convergence for $\la e, x \ra$.  For each $i < b$, let $P_i$ be the tree consisting of the strings in $S'$ that are extendible to an $\alpha \in \leaves(S')$ such that $\Phi_e^{f \oplus \alpha}(x) = i$, and let $P_b$ be the tree consisting of the strings in $S'$ that are extendible to an $\alpha \in \leaves(S')$ such that $\Phi_e^{f \oplus \alpha}(x) \geq b$.  Then $S' \subseteq \bigcup_{i < b+1}P_i$, so by Lemma~\ref{lem-CoverWithNTrees} there is a $\geq\! m2^{|D|}$-branching tree $S \subseteq S'$ with trunk $\sigma'$ such that $S \subseteq P_i$ for some $i < b+1$.  If $S \subseteq P_b$, then the construction would have acted on $T_{t'}(\sigma')$ according to item~(b) part~(ii) at some stage $t' > s$, contradicting the choice of $s$.  Thus $S \subseteq P_i$ for some $i < b$.  Thus there are indeed a stage $t > s$ with $t,\ell_{e,t} > x$, an $i < b$, and a $\geq\! m2^{|D|}$-branching tree $S \subseteq T_t(\sigma')$ such that $(\forall \alpha \in \leaves(S))(\Phi_e^{f \oplus \alpha}(x) = i)$.  So $\Phi^f$ is total.  To see that $\Phi^f(x) \neq \Phi_x^h(x)$ for all $x > |\sigma|$, suppose for a contradiction that $x > |\sigma|$ is such that $\Phi^f(x) = \Phi_x^h(x)$.  By the definition of $\Phi^f(x)$, there are a stage $t > s$ and a $\geq\! m2^{|D|}$-branching tree $S \subseteq T_t(\sigma')$ such that $(\forall \alpha \in \leaves(S))(\Phi_e^{f \oplus \alpha}(x) = \Phi_x^h(x))$.  Then the construction would have acted on $T_{t'}(\sigma')$ according to item~(b) part~(ii) at some stage $t' > s$, contradicting the choice of $s$.  Thus $\Phi^f$ is eventually $\dnrstat(b,h)$, contradicting that $f$ computes no such function.  Therefore we cannot have $\lim_{t \imp \infty} \ell_{e,t} = \infty$, hence $(\exists t)(\forall t' > t)(\ell_{e,t'} = \ell_{e,t})$ as desired.
\end{proof}

Applying $\bst$ to the claim, we have that, in fact, $(\exists t_0 > s)(\forall e \in E \setminus M(\sigma'))(\forall t' > t_0)(\ell_{e,t'} = \ell_{e,t_0})$.  For each $e \in E \setminus M(\sigma')$, let $\ell_e = \ell_{e,t_0}$.  Then let $D' = D \cup \{\la e, \ell_e \ra : e \in E \setminus M(\sigma')\}$.  We show that $\sigma'$ and $D'$ satisfy the conclusion of the lemma.  The inequality $|D'| \leq |D| + |E|$ is clear.

First, $\sigma'$ forces $\geq\! m$-branching $f$-divergence (and hence $\geq\! m2^{|D|+b}$-branching $f$-divergence) for $D$, otherwise the construction would act according to item~(a) part~(i) at some stage past $s$.  To see that $\sigma'$ forces $\geq\! m2^{|D|+b}$-branching $f$-divergence for each of the $\la e, \ell_e \ra$ with $e \in E \setminus M(\sigma')$, suppose not and let $\la e, \ell_e \ra$, with $e \in E \setminus M(\sigma')$, and $S$, a tree with trunk $\sigma'$, be such that $S$ witnesses that $\sigma'$ admits $\geq\! m2^{|D|+b}$-branching $f$-convergence for $\la e, \ell_e \ra$.  As the construction never acts on $T_t(\sigma')$ for $t > s$, there is a stage $t > t_0$ with $S \subseteq T_t(\sigma')$.  Thus at stage $t+1$ there is a tree $S \subseteq T_t(\sigma')$ with trunk $\sigma'$ witnessing that $\sigma'$ admits $\geq\! m2^{|D|+b}$-branching $f$-convergence for $\la e, \ell_e \ra = \la e, \ell_{e,t} \ra$, contradicting the choice of $\ell_{e,t}$.

Finally, we show that, for each $e \in E$, either $(\exists \ell > |\sigma|)(\Phi_e^{f \oplus \sigma'}(\ell) \geq b)$, $(\exists \ell > |\sigma|)(\Phi_e^{f \oplus \sigma'}(\ell) = \Phi_\ell^h(\ell))$, or $\exists \ell(\la e, \ell \ra \in D')$.  By the definition of $D'$, if $e \in E \setminus M(\sigma')$ then there is an $\ell$ such that $\la e, \ell \ra \in D'$.  Thus we need to show that if $e \in M(\sigma')$ then either $(\exists \ell > |\sigma|)(\Phi_e^{f \oplus \sigma'}(\ell) \geq b)$ or $(\exists \ell > |\sigma|)(\Phi_e^{f \oplus \sigma'}(\ell) = \Phi_\ell^h(\ell))$.  Suppose that $e \in M(\sigma')$, and let $t+1 \leq s$ be least such that $e \in M(\eta)$ for some $\eta \subseteq \sigma'$ with $\eta \in R_{t+1}$.  Then $e$ entered $M(\eta)$ at stage $t+1$ by an action according to item~(b) part~(ii).  Thus at stage $t+1$ there must have been a $\tau \in R_t$ with $\eta \in \leaves(T_t(\tau))$ and a least $x$ with $|\sigma| < x \leq t$ having a least $\geq\! m2^{|D|}$-branching tree $S \subseteq T_t(\tau)$ with trunk $\tau$ such that either $(\forall \alpha \in \leaves(S))(\Phi_e^{f \oplus \alpha}(x) = \Phi_{x,t}^h(x))$ or $(\forall \alpha \in \leaves(S))(\Phi_e^{f \oplus \alpha}(x) \geq b)$.  Moreover, $\eta$ must extend a leaf of $S$ because $T_{t+1}(\eta)$ must be alive because $\eta$ is an initial segment of a path through $T$.  So if $(\forall \alpha \in \leaves(S))(\Phi_e^{f \oplus \alpha}(x) = \Phi_{x,t}^h(x))$, then $(\exists \ell > |\sigma|)(\Phi_e^{f \oplus \sigma'}(\ell) = \Phi_\ell^h(\ell))$, and if $(\forall \alpha \in \leaves(S))(\Phi_e^{f \oplus \alpha}(x) \geq b)$ then $(\exists \ell > |\sigma|)(\Phi_e^{f \oplus \sigma'}(\ell) \geq b)$.
\end{proof}

\begin{Lemma}[$\rcasys$]\label{lem-Low}
Let $f$ be a function, let $\sigma$ be a string, and let $D$ be a finite coded set such that $\sigma$ forces $\geq\! m$-branching $f$-divergence for $D$, where $m2^{|D|} < k$.  Then for every finite coded set $E$ there is a string $\sigma' \supseteq \sigma$ with the following property.  Let $E' = \{e \in E : \Phi_e^{f \oplus \sigma'}(e)\ua\}$, and let $e'$ be an index for a program such that, for any function $g$, $(\Phi_{e'}^g(e')\da) \biimp (\exists e \in E')(\Phi_e^g(e)\da)$.  Then $\sigma'$ forces $\geq\! m2^{|D|}$-branching $f$-divergence for $D \cup \{\la e', e' \ra\}$.
\end{Lemma}

\begin{proof}
The proof is similar to that of Lemma~\ref{lem-AvoidDNR2}.  We prove the lemma in $\wklsys$, which suffices because $\wklsys$ is $\Pi^1_1$-conservative over $\rcasys$ (see \cite{Simpson:2009vv}~Corollary~IX.2.7).  Thus we construct an infinite tree $T$ with trunk $\sigma$ such that every infinite path through $T$ has an initial segment $\sigma'$ that satisfies the conclusion of the lemma.

As in Lemma~\ref{lem-AvoidDNR2}, $T$ grows in stages $(T_s : s \in \Nb)$, where $T_s = \bigcup_{\tau \in R_s} T_s(\tau)$ and, for each $\tau \in R_s$, $T_s(\tau)$ is $\tau^\smf k^{<t}$ for some $t \in \Nb$.  The component trees are either alive or dead, as before.  To every $s \in \Nb$ and $\tau \in R_s$ we associate the set $E(\tau) = \{e \in E : \Phi_e^{f \oplus \tau}(e)\ua\}$ and the index $e(\tau)$, where $(\Phi_{e(\tau)}^g(e(\tau))\da) \biimp (\exists e \in E(\tau))(\Phi_e^g(e)\da)$.

At stage $0$, let $R_0 = \{\sigma\}$ and $T_0(\sigma) = \sigma$.  $T_0(\sigma)$ is alive at stage $0$.

At the beginning of stage $s+1$ we have $T_s$ represented as $T_s = \bigcup_{\tau \in R_s} T_s(\tau)$, and we have the corresponding auxiliary information $(E(\tau) : \tau \in R_s)$ and $(e(\tau) : \tau \in R_s)$.  For each $\tau \in R_s$ do the following:

\begin{itemize}
\item[(a)] If $T_s(\tau)$ is dead, put $\tau$ in $R_{s+1}$ and let $T_{s+1}(\tau) = T_s(\tau)$.  $T_{s+1}(\tau)$ is dead.

\item[(b)] If $T_s(\tau)$ is alive:
\begin{itemize}
\item[(i)] If there are a $t \leq s$, a $\tau' \in R_t$ with $\tau' \subseteq \tau$, an $\la e, x \ra \in D$, and a $\geq\! m$-branching tree $S \subseteq k^{< s}$ with trunk $\tau'$ such that $(\forall \alpha \in \leaves(S))(\Phi_e^{f \oplus \alpha}(x)\da)$ (i.e., we learn at stage $s$ that $\tau'$ admits $\geq\! m$-branching $f$-convergence for $\la e, x \ra \in D$), then put $\tau$ in $R_{s+1}$ and let $T_{s+1}(\tau) = T_s(\tau)$.  $T_{s+1}(\tau)$ is dead.

\item[(ii)] If (i) fails and there is a $\geq\! m2^{|D|}$-branching tree $S \subseteq T_s(\tau)$ with trunk $\tau$ such that $(\forall \alpha \in \leaves(S))(\Phi_{e(\tau)}^{f \oplus \alpha}(e(\tau))\da)$, then choose the first such $S$.  Put all leaves of $T_s(\tau)$ in $R_{s+1}$, and let $T_{s+1}(\eta) = \eta$ for all $\eta \in \leaves(T_s(\tau))$.  If $\eta \in \leaves(T_s(\tau))$ extends a leaf of $S$, then $T_{s+1}(\eta)$ is alive; otherwise $T_{s+1}(\eta)$ is dead.
\end{itemize}
\end{itemize}
Finally, for each $\tau \in R_{s+1}$ with $T_{s+1}(\tau)$ alive, grow $T_{s+1}(\tau)$ by extending each $\alpha \in \leaves(T_{s+1}(\tau))$ to $\alpha^\smf n$ for every $n < k$.  This concludes stage $s+1$.

The tree $T$ is an infinite $f$-recursive subtree of $k^{< \Nb}$ by arguments similar to those in the proof of Lemma~\ref{lem-AvoidDNR2}.  By $\wklsys$, let $p$ be an infinite path through $T$.  Using bounded $\Sigma^0_1$ comprehension, let $s \in \Nb$ and $\sigma' \subset p$ minimize $|E(\sigma')|$ over all $s \in \Nb$ and $\sigma' \in R_s$ with $\sigma' \subset p$.  Then $\sigma'$ satisfies the conclusion of the lemma.  Note that the corresponding $e'$ is $e(\sigma')$.  As in the proof of Lemma~\ref{lem-AvoidDNR2}, the construction never acts on $T_t(\sigma')$ according to item~(b) at any stage $t > s$.  Consequently, $\sigma'$ forces $\geq\! m$-branching $f$-divergence (and hence $\geq\! m2^{|D|}$-branching $f$-divergence) for $D$ because otherwise the construction would act according to item~(b) part~(i) at some stage past $s$.  Similarly, $\sigma'$ forces $\geq\! m2^{|D|}$-branching $f$-divergence for $\la e', e' \ra = \la e(\sigma'), e(\sigma') \ra$ because otherwise the construction would act on $T_t(\sigma')$ according to item~(b) part~(ii) at some stage $t > s$.  Thus $\sigma'$ forces $\geq\! m2^{|D|}$-branching $f$-divergence for $D \cup \{\la e', e' \ra\}$.
\end{proof}

\begin{proof}[Proof of Theorem~\ref{thm-DNRKwithoutDNR2}]
Let $M$, $I$, $n$, $f$, $\vec h$, $\vec b$, $b_{\max}$, $k_0$, and $k$ be as in the statement of Theorem~\ref{thm-DNRKwithoutDNR2}.  For each $i < n$, fix an index $w_i$ such that $\forall x(\Phi_{w_i}^f(x) = \Phi^{h_i}_x(x))$.  The proof of Lemma~\ref{lem-cut} shows that there is an increasing, cofinal function $c \colon I \imp \Nb$ whose graph is $\Delta^0_2$.

We build a $\Delta^0_2(f)$ sequence $(\la \sigma_s, D_s \ra : s \in J)$ in stages, where $J \subseteq I$ is a $\Sigma^0_2(f)$ cut determined during the course of the construction.  In the end, we set $g = \bigcup_{s \in J}\sigma_s$.  Let $e_0$ be an index such that, for any $g$ and $x$, $(\Phi_{e_0}^{f \oplus g}(x)\da) \biimp \exists e(g(e) = \Phi_e^f(e))$.  At stage $0$, set $\sigma_0 = \emptyset$ and set $D_0 = \{\la e_0, e_0 \ra\}$.
\begin{itemize}
\item At stage $s+1 \equiv 0 \mod n+2$, search for the least $\sigma_{s+1} \supseteq \sigma_s$ such that $|\sigma_{s+1}| > \max\{c(s),D_s\}$ (where here $D_s$ is interpreted as the number coding the set $D_s$) and that $\sigma_{s+1}$ forces $\geq\! K(b_{\max},s+1)$-branching $f$-divergence for $D_s$.  Let $D_{s+1} = D_s$.

\item At stage $s+1 \equiv i+1 \mod n+2$ for an $i < n$, search for the least pair $\la \sigma', D' \ra$ as in the conclusion of Lemma~\ref{lem-AvoidDNR2} for $f$, $h = h_i$, $b = b_i$, $\sigma = \sigma_s$, $D = D_s$, $m = K(b_{\max},s)$, and $E = \{d(w_i, |\sigma_t|, b_i) : t \leq s\}$.  Let $\sigma_{s+1} = \sigma'$ and let $D_{s+1} = D'$.

\item At stage $s+1 \equiv n+1 \mod n+2$, search for the least $\sigma'$ as in the conclusion of Lemma~\ref{lem-Low} for $f$, $\sigma = \sigma_s$, $D = D_s$, $m=K(b_{\max}, s)$, and $E = \{t : t \leq |\sigma_s|\}$.  Let $\sigma_{s+1} = \sigma'$ and let $D_{s+1} = D \cup \{\la e', e' \ra\}$.
\end{itemize}

Let $J$ be the set of $s \in \Nb$ such that the construction reaches stage $s$.  That is, $J$ is the set of $s \in \Nb$ for which there is a sequence $(\la \sigma_j, D_j \ra : j \leq s)$ where $\sigma_0 = \emptyset$, $D_0 = \{\la e_0, e_0 \ra\}$, and, for all $j<s$, $\la \sigma_{j+1}, D_{j+1} \ra$ follows from $\la \sigma_j, D_j \ra$ according to the rules of the construction.  Checking whether $\la \sigma_{j+1}, D_{j+1} \ra$ follows from $\la \sigma_j, D_j \ra$ is $\Delta^0_2(f)$, so $J$ is $\Sigma^0_2(f)$.

Clearly $J$ is downward closed.  To see $J \subseteq I$, let $s \in J$ and let $n_0 < n+2$ be such that $s-n_0 \equiv 0 \mod n+2$.  Then $s-n_0$ must be in $I$ because $c(s-n_0)$ must be defined in order for $s-n_0$ to be in $J$.  Hence $s \in I$ because $I$ is a cut and $n_0 \in \omega$.

Notice that at stage $s+1$ at most $s+1$ elements are added to $D_{s+1}$.  Therefore, for all $s \in J$, $|D_s| \leq 1+ \sum_{j \leq s} j = \frac{1}{2}(s^2+s) + 1 \leq s^2 + 1$ (the `$+1$' is because $|D_0| = 1$, not $|D_0| = 0$).

\begin{Claim}
If $s \in J$, then $\sigma_s$ forces $\geq\! K(b_{\max}, s)$-branching $f$-divergence for $D_s$.
\end{Claim}

\begin{proof}[Proof of Claim]
Let $(\la \sigma_j, D_j \ra : j \leq s)$ be a witness to $s \in J$.  We prove the claim by $\Pi^0_1$ induction on $j \leq s$.  To see that $\sigma_0 = \emptyset$ forces $\geq\! 2$-branching $f$-divergence for $D_0 = \{\la e_0, e_0 \ra\}$, consider a $\geq\! 2$-branching tree $T$ with trunk $\emptyset$.  Let $t$ be the height of $T$, and, by bounded $\Sigma^0_1$ comprehension, let $X = \{e < t : \Phi_e^f(e)\da\}$.  As $T$ is $\geq\! 2$-branching, we can find an $\alpha \in \leaves(T)$ such that $(\forall e \in X)(\alpha(e) \neq \Phi_e^f(e))$.  Then $\Phi_{e_0}^{f \oplus \alpha}(e_0)\ua$, showing that $T$ does not witness that $\emptyset$ admits $\geq\! 2$-branching $f$-convergence for $\la e_0, e_0 \ra$, as desired.

Now suppose that $j < s$ and that $\sigma_j$ forces $\geq\! K(b_{\max}, j)$-branching $f$-divergence for $D_j$.

\begin{itemize}
\item If $j+1 \equiv 0 \mod n+2$, then $\sigma_{j+1}$ forces $\geq\! K(b_{\max}, j+1)$-branching $f$-divergence for $D_{j+1}$ by definition.

\item If $j+1 \equiv i+1 \mod n+2$ for an $i < n$, then $\sigma_{j+1}$ forces $\geq\! K(b_{\max}, j)2^{|D_j|+b_i}$-branching $f$-divergence for $D_{j+1}$ by definition (refer to the statement of Lemma~\ref{lem-AvoidDNR2}).  As $|D_j| + b_i \leq j^2 + b_{\max} + 1$, $\sigma_{j+1}$ forces $\geq\! K(b_{\max}, j+1)$-branching $f$-divergence for $D_{j+1}$.

\item If $j+1 \equiv n+1 \mod n+2$, then $\sigma_{j+1}$ forces $\geq\! K(b_{\max}, j)2^{|D_j|}$-branching $f$-divergence for $D_{j+1}$ by definition (refer to the statement of Lemma~\ref{lem-Low}).  As $|D_j| \leq j^2 + b_{\max} + 1$, $\sigma_{j+1}$ forces $\geq\! K(b_{\max}, j+1)$-branching $f$-divergence for $D_{j+1}$.
\end{itemize}
\end{proof}

\begin{Claim}
$J$ is a cut.
\end{Claim}

\begin{proof}[Proof of Claim]
We have seen that $J$ is downward closed.  We need to show that $\forall s(s \in J \imp s+1 \in J)$.  So suppose $s \in J$.  By the previous claim, $\sigma_s$ forces $\geq\! K(b_{\max}, s)$-branching $f$-divergence for $D_s$.

\begin{itemize}
\item If $s+1 \equiv 0 \mod n+2$, then consider the tree ${\sigma_s}^\smf K(b_{\max}, s+1)^{<\max\{c(s),D_s\}}$.  It is $\geq\! K(b_{\max}, s)2^{|D_s|}$-branching with trunk $\sigma_s$, so by Lemma~\ref{lem-PreserveDivergence} it has a leaf that forces $\geq\! K(b_{\max}, s)$-branching $f$-divergence for $D_s$, and this leaf also forces $\geq\! K(b_{\max}, s+1)$-branching $f$-divergence for $D_s$.  Thus $\sigma_{s+1}$ and $D_{s+1}$ are defined.

\item If $s+1 \equiv i+1 \mod n+2$ for an $i < n$, then Lemma~\ref{lem-AvoidDNR2} applies.  The previous claim and the inequality $K(b_{\max},s)2^{|D_s|+b} < k$ show that the hypotheses of Lemma~\ref{lem-AvoidDNR2} are satisfied, where $K(b_{\max},s)2^{|D_s|+b} < k$ because
\begin{align*}
K(b_{\max},s)2^{|D_s|+b} \leq K(b_{\max},s)2^{s^2+b+1} = K(b_{\max},s+1) < K(b_{\max},k_0) = k,
\end{align*}
with the strict inequality holding because $s+1 \in I$ and therefore $s+1 < k_0$.  Thus $\sigma_{s+1}$ and $D_{s+1}$ are defined.

\item If $s+1 \equiv n+1 \mod n+2$, then Lemma~\ref{lem-Low} applies by an argument similar to the one in the previous item.  Thus $\sigma_{s+1}$ and $D_{s+1}$ are defined.  
\end{itemize}
\end{proof}

\begin{Claim}
The function $g$ is total.
\end{Claim}

\begin{proof}[Proof of Claim]
If $g$ is not total, then there is a $t \in \Nb$ such that $(\forall s \in J)(|\sigma_s| < t \andd D_s < t)$.  By increasing $t$, we may assume that if $s \in J$ then $\sigma_s < t$ (i.e., $\sigma_s$ is coded by a number $< t$).  Thus $s \in J$ if and only if there is a sequence $(\la \sigma_j, D_j \ra : j \leq s) \leq \la t,t \ra^{s+1}$ (i.e., the sequence of $s+1$ copies of $\la t,t \ra$), where $\sigma_0 = \emptyset$, $D_0 = \la e_0, e_0 \ra$, and, for all $j<s$, $\la \sigma_{j+1}, D_{j+1} \ra$ follows from $\la \sigma_j, D_j \ra$ according to the rules of the construction.  This shows that $J$ is $\Delta^0_2(f)$, which is a contradiction because by $\bst$ there are no $\Delta^0_2(f)$ cuts.
\end{proof}

\begin{Claim}
The function $g$ is $\dnrstat(k,f)$.
\end{Claim}

\begin{proof}[Proof of Claim]
The function $g$ has range contained in $k$ by the convention that all trees are subtrees of $k^{< \Nb}$.  Suppose for a contradiction that $\exists e (g(e) = \Phi_e^f(e))$.  Then $\Phi_{e_0}^{f \oplus g}(e_0)\da$, so there is an initial segment $\sigma \subset g$ such that $\Phi_{e_0}^{f \oplus \sigma}(e_0)\da$.  Let $s$ be a stage with $\sigma_s \supseteq \sigma$.  Then $\Phi_{e_0}^{f \oplus \sigma_s}(e_0)\da$, but this is a contradiction because $\la e_0, e_0 \ra \in D_s$ and $\sigma_s$ forces $\geq\! K(b_{\max}, s)$-branching $f$-divergence for $D_s$.
\end{proof}

\begin{Claim}
$f \oplus g$ is low relative to $f$.
\end{Claim}

\begin{proof}[Proof of Claim]
To determine whether or not $\Phi^{f \oplus g}_e(e)\da$, run the construction to a stage $s$ with $s+1 \equiv n+1 \mod n+2$ and $|\sigma_s| > e$.  Then $\Phi_e^{f \oplus g}(e)\da$ if and only if $\Phi_e^{f \oplus \sigma_{s+1}}(e)\da$.  Since the sequence $(\la \sigma_s, D_s \ra : s \in J)$ is $\Delta^0_2(f)$, it then follows that $f \oplus g$ is low relative to $f$.  Clearly $(\Phi_e^{f \oplus \sigma_{s+1}}(e)\da) \imp (\Phi_e^{f \oplus g}(e)\da)$.  To see $(\Phi_e^{f \oplus \sigma_{s+1}}(e)\ua) \imp (\Phi_e^{f \oplus g}(e)\ua)$, suppose for a contradiction that $\Phi_e^{f \oplus \sigma_{s+1}}(e)\ua$ but $\Phi_e^{f \oplus g}(e)\da$.  Let $\la e', e' \ra$ be the element added to $D_{s+1}$ at stage $s+1$.  Then $\Phi_{e'}^{f \oplus g}(e')\da$ because $e$ is in $E'$ (where $E'$ is as in Lemma~\ref{lem-Low}) and $\Phi_e^{f \oplus g}(e)\da$.  Let $r > s$ be a stage such that $\Phi_{e'}^{f \oplus \sigma_r}(e')\da$.  Clearly $\sigma_r$ admits $\geq\! K(b_{\max}, r)$-branching $f$-convergence for $\la e', e' \ra$, but this contradicts that $\la e', e' \ra \in D_r$ and $\sigma_r$ forces $\geq\! K(b_{\max}, r)$-branching $f$-divergence for $D_r$.
\end{proof}

\begin{Claim}
For each $i < n$, the function $f \oplus g$ computes no $\dnrstat(b_i, h_i)$ function.
\end{Claim}

\begin{proof}[Proof of Claim]
Suppose for a contradiction that $\Phi_e^{f \oplus g}$ is $\dnrstat(b_i, h_i)$.  Fix a stage $t$ such that $|\sigma_t| > e$.  By the previous claim, $\bst$ and hence $\iso$ holds relative to $f \oplus g$.  Thus Lemma~\ref{lem-DNRblock} applies to $f \oplus g$, so $\Phi_{d(w_i, |\sigma_t|, b_i)}^{f \oplus g}$ is eventually $\dnrstat(b_i, h_i)$.  We show that in fact $\Phi_{d(w_i, |\sigma_t|, b_i)}^{f \oplus g}$ is not eventually $\dnrstat(b_i, h_i)$, giving the contradiction.  Fix $\ell_0$.  We want to find an $\ell > \ell_0$ such that $\Phi_{d(w_i, |\sigma_t|, b_i)}^{f \oplus g}(\ell) \geq b_i$ or $\Phi_{d(w_i, |\sigma_t|, b_i)}^{f \oplus g}(\ell) = \Phi_{\ell}^{h_i}(\ell)$.  Let $s+1 > t$ be a stage with $s+1 \equiv i+1 \mod n+2$ and $|\sigma_s| > \ell_0$.  At stage $s+1$, $\sigma_{s+1}$ and $D_{s+1}$ are defined to be as in the conclusion of Lemma~\ref{lem-AvoidDNR2} for an $E$ with $d(w_i, |\sigma_t|, b_i) \in E$.  The result is that $(\exists \ell > |\sigma_s|)(\Phi_{d(w_i, |\sigma_t|, b_i)}^{f \oplus \sigma_{s+1}}(\ell) \geq b_i)$, $(\exists \ell > |\sigma_s|)(\Phi_{d(w_i, |\sigma_t|, b_i)}^{f \oplus \sigma_{s+1}}(\ell) = \Phi_\ell^{h_i}(\ell))$, or $\exists \ell(\la d(w_i, |\sigma_t|, b_i), \ell \ra \in D_{s+1})$.  If either of the first two alternatives hold, then we have our $\ell > \ell_0$ such that $\Phi_{d(w_i, |\sigma_t|, b_i)}^{f \oplus g}(\ell) \geq b_i$ or $\Phi_{d(w_i, |\sigma_t|, b_i)}^{f \oplus g}(\ell) = \Phi_{\ell}^{h_i}(\ell)$.  If the third alternative holds, then $\Phi_{d(w_i, |\sigma_t|, b_i)}^{f \oplus g}(\ell)\ua$, again contradicting that $\Phi_{d(w_i, |\sigma_t|, b_i)}^{f \oplus g}$ is eventually $\dnrstat(b_i,h_i)$.  To see that $\Phi_{d(w_i, |\sigma_t|, b_i)}^{f \oplus g}(\ell)\ua$, suppose instead that $\Phi_{d(w_i, |\sigma_t|, b_i)}^{f \oplus g}(\ell)\da$ and let $r > s+1$ be a stage such that $\Phi_{d(w_i, |\sigma_t|, b_i)}^{f \oplus \sigma_r}(\ell)\da$.  Clearly $\sigma_r$ admits $\geq\! K(b_{\max}, r)$-branching $f$-convergence for $\la d(w_i, |\sigma_t|, b_i), \ell \ra$, but this contradicts that $\la d(w_i, |\sigma_t|, b_i), \ell \ra \in D_r$ and $\sigma_r$ forces $\geq\! K(b_{\max}, r)$-branching $f$-divergence for $D_r$.
\end{proof}

This concludes the proof of Theorem~\ref{thm-DNRKwithoutDNR2}.
\end{proof}

\begin{Theorem}\label{thm-noDNRK}
$\rcasys + \bst + \forall f \exists k \dnrstat(k,f) \nvdash \exists k \forall f \dnrstat(k,f)$.
\end{Theorem}

\begin{proof}
We build a model of $\rcasys + \bst + \forall f \exists k \dnrstat(k,f) + \neg\exists k \forall f \dnrstat(k,f)$ by iterating Theorem~\ref{thm-DNRKwithoutDNR2}.

Let $\Nb$ be a countable first-order model of $\bst + \neg\ist$.  By Lemma~\ref{lem-cut}, let $I$ be a proper $\Sigma^0_2$ cut in $\Nb$.  Fix $k_0 \in \Nb$ such that $(\forall i \in I)(k_0 > i)$.

Fix an increasing, cofinal sequence $(b_m : m \in \omega)$ of numbers in $\Nb$.  We define a sequence $(f_m : m \in \omega)$ of functions $\Nb \imp \Nb$ such that, for all $m \in \omega$,
\begin{itemize}
\item[(i)] $f_m \leqT f_{m+1}$;

\item[(ii)] $(\Nb, \Delta^0_1(f_m)) \models \rcasys + \bst$;

\item[(iii)] no $h \leqT f_m$ is $\dnrstat(b_{m_0},f_{m_0})$ for any $m_0 \leq m$;

\item[(iv)] for every $h \leqT f_m$, there are a $k \in \Nb$ and a $g \leqT f_{m+1}$ that is $\dnrstat(k,h)$.
\end{itemize}
Let $f_0 = 0$.  The function $f_0$ is $\Delta^0_1$, so items~(ii) and~(iii) hold for $m=0$, with item~(ii) holding because $\Nb \models \bst$.  Suppose now that $(f_j : j < m+1)$ satisfies items~(i) and~(iv) for all $j < m$ and satisfies items~(ii) and~(iii) for all $j < m+1$.  Then $M = (\Nb, \Delta^0_1(f_m))$, $I$, $n = m+1$, $f = f_m$, $\vec h = (f_j : j < m+1)$, $\vec b = (b_j : j < m+1)$, and $k = K(b_{\max}, k_0)$ satisfy the hypotheses of Theorem~\ref{thm-DNRKwithoutDNR2}.  Thus let $g$ be as in the conclusion of Theorem~\ref{thm-DNRKwithoutDNR2}, and let $f_{m+1} = f_m \oplus g$.  Then item~(i) holds for $m$ and items~(ii) and~(iii) hold for $m+1$, with item~(ii) holding because $f_{m+1}$ is low relative to $f_m$.  Item~(iv) holds for $m$ because $g \leqT f_{m+1}$ is $\dnrstat(k,f_m)$ and hence computes a $\dnrstat(k,h)$ function for every $h \leqT f_m$.

Let $\MP S = \bigcup_{m \in \omega}\Delta^0_1(f_m)$.  Then $(\Nb, \MP S)$ models $\rcasys + \bst + \forall f \exists k \dnrstat(k,f) + \neg\exists k \forall f \dnrstat(k,f)$.  $\bst$ holds relative to every $h \in \MP S$ by item~(ii), so $(\Nb, \MP S) \models \rcasys + \bst$.  We have that $(\Nb, \MP S) \models \forall f \exists k \dnrstat(k,f)$ by item~(iv).  To see that $(\Nb, \MP S) \not\models \exists k \forall f \dnrstat(k,f)$, let $k \in \Nb$ and let $b_{m_0} > k$.  Then observe that no $h \in \MP S$ is $\dnrstat(b_{m_0}, f_{m_0})$ (hence no $h \in \MP S$ is $\dnrstat(k, f_{m_0})$) by item~(iii).
\end{proof}

\begin{Theorem}\label{thm-noWKL}
$\rcasys + \bst + \exists k \forall f \dnrstat(k,f) \nvdash \wklstat$.
\end{Theorem}

\begin{proof}
The proof is a simplification of the proof of Theorem~\ref{thm-noDNRK}.  We build a model of $\rcasys + \bst + \exists k \forall f \dnrstat(k,f) + \neg\forall f \dnrstat(2,f)$ by iterating Theorem~\ref{thm-DNRKwithoutDNR2}.  As $\forall f \dnrstat(2,f)$ and $\wklstat$ are equivalent over $\rcasys$, this also a model of $\rcasys + \bst + \exists k \forall f \dnrstat(k,f) + \neg \wklstat$.

Proceed as in the proof of Theorem~\ref{thm-noDNRK}, but fix $k = K(2,k_0)$ and ignore the sequence $(b_m : m \in \omega)$.  Define a sequence $(f_m : m \in \omega)$ of functions $\Nb \imp \Nb$ that satisfy items~(i) and~(ii) as before and satisfy the following modified versions of items~(iii) and~(iv):
\begin{itemize}
\item[(iii')] no $h \leqT f_m$ is $\dnrstat(2,0)$;

\item[(iv')] for every $h \leqT f_m$, there is a $g \leqT f_{m+1}$ that is $\dnrstat(k,h)$.
\end{itemize}
Now $f_{m+1}$ is obtained from $f_m$ by applying Theorem~\ref{thm-DNRKwithoutDNR2} to $M = (\Nb, \Delta^0_1(f_m))$, $I$, $n=1$, $f = f_m$, $\vec h = (0)$, $\vec b = (2)$, and $k$.  The witnessing model $(\Nb, \MP S)$ is built from $(f_m : m \in \omega)$ as before.
\end{proof}

Now that we know that the statements $\exists k \forall f \dnrstat(k,f)$ and $\forall f \exists k \dnrstat(k,f)$ do not imply $\wklstat$ even over $\rcasys + \bst$, it is natural to ask if either statement implies weak weak K\"onig's lemma.

\begin{Question}
Do either $\exists k \forall f \dnrstat(k,f)$ or $\forall f \exists k \dnrstat(k,f)$ imply $\wwklstat$ over $\rcasys$ (or over $\rcasys + \bst$)?
\end{Question}

\section{Observations concerning the connection between diagonally non-recursive functions and graph colorings}

Just as $\dnrstat(\ell,f)$ trivially implies $\dnrstat(k,f)$ when $k \geq \ell$, so the existence of an $\ell$-coloring of a graph trivially implies the existence of a $k$-coloring of that graph when $k \geq \ell$.  This motivates the search for a connection between $\dnrstat$ functions and graph colorings.  So far, our efforts in this area have produced more questions than answers.

\begin{Definition}[$\rcasys$]
A \emph{graph} $G = (V,E)$ consists of a set of vertices $V \subseteq \Nb$ and an irreflexive, symmetric relation $E \subseteq V \times V$ which indicates when two vertices are adjacent.  Let $G$ be a graph, and let $\ell \in \Nb$.
\begin{itemize}
\item An \emph{$\ell$-coloring} of $G$ is a function $\chi \colon V \imp \ell$ such that $(\forall u,v \in V)((u,v) \in E \imp \chi(u) \neq \chi(v))$.

\item $G$ is (\emph{globally}) \emph{$\ell$-colorable} if there is an $\ell$-coloring of $G$.

\item $G$ is \emph{locally $\ell$-colorable} if for every finite $V_0 \subseteq V$, the induced subgraph $(V_0, E \cap (V_0 \times V_0))$ is $\ell$-colorable.
\end{itemize}
\end{Definition}

Let $\colstat(\ell,k,G)$ denote the formal statement that ``if the graph $G$ locally $\ell$-colorable, then $G$ is globally $k$-colorable.'' A classic compactness argument shows that a graph is $\ell$-colorable if and only if it is locally $\ell$-colorable. In the context of reverse mathematics, the following theorem expresses that this fact is equivalent to $\wklstat$ over $\rcasys$.

\begin{Theorem}[see~\cite{Hirst:1987}~Theorem 3.4]\label{thm-h}
\begin{align*}
\rcasys \vdash (\forall \ell \geq 2)(\wklstat \biimp \forall G\,\colstat(\ell,\ell,G))).
\end{align*}
\end{Theorem}

In~\cite{Bean:1976ta}, Bean gave an example of a recursive $3$-colorable graph that has no recursive $k$-coloring for any $k \in \omega$. This result suggests that coloring a $3$-colorable (or more generally $\ell$-colorable) graph with any finite number of colors may also be difficult from the proof-theoretic point of view.  To this end, Gasarch and Hirst proved the following theorem.

\begin{Theorem}[\cite{Gasarch:1998gy}~Theorem~3]\label{thm-gh}
\begin{align*}
\rcasys \vdash (\forall \ell \geq 2)(\wklstat \biimp \forall G\,\colstat(\ell,2\ell-1,G)).
\end{align*}
\end{Theorem}

Gasarch and Hirst then conjectured that the $(2\ell-1)$ in their theorem can be replaced by any $k \geq \ell$.

\begin{Conjecture}[\cite{Gasarch:1998gy}~Conjecture~4]\label{conj-gh}
\begin{align*}
\rcasys \vdash (\forall \ell \geq 2)(\forall k \geq \ell)(\wklstat \biimp \forall G\,\colstat(\ell,k,G)).
\end{align*}
\end{Conjecture}

In~\cite{Schmerl:2000el}, Schmerl verified a weakened version of this conjecture in which $\ell$ and $k$ are both fixed and standard.

\begin{Theorem}[\cite{Schmerl:2000el}~Theorem~1]\label{thm-Schmerl}
Fix $k,\ell \in \omega$ with $k \geq \ell \geq 2$.  Then 
\begin{align*}
\rcasys \vdash \wklstat \biimp \forall G\,\colstat(\ell,k,G).
\end{align*}
\end{Theorem}

Schmerl connected two key ingredients to prove Theorem~\ref{thm-Schmerl}.  The first ingredient is the \emph{on-line coloring game} $\Gamma_d(\cls K,k)$, where $\cls K$ is a class of graphs and $d,k \in \Nb$.  $\Gamma_d(\cls K,k)$ is a game between two players, $\forall$ and $\exists$.  Player $\forall$ builds a graph in $\cls K$, and Player $\exists$ $k$-colors it.  The game lasts for $d$ rounds.  In each round, $\forall$ and $\exists$ alternate plays as follows.  Player $\forall$ goes first by adding a new vertex to the graph and connecting it to the existing vertices in such a way that the graph remains in $\cls K$.  Player $\exists$ goes second and colors the new vertex with a color from $\{0,1 ,\dots, k-1\}$.  After $d$ rounds, $\exists$ wins if she has produced a $k$-coloring of the graph enumerated by $\forall$.  We say that the class $\cls K$ is \emph{locally on-line $k$ colorable} if for every $d \in \Nb$, $\exists$ has a winning strategy in $\Gamma_d(\cls K,k)$.  (Assuming $\wklstat$, it then follows that $\exists$ has a winning strategy in the unbounded on-line coloring game $\Gamma(\cls K,k)$, where the two players continue for as long as $\forall$ keeps playing new vertices.)

The second ingredient can be found in \cite{Schmerl:2000el}~Lemma~2.3, where Schmerl isolates a recursion-theoretic principle similar to the negation of $\dnrstat(k,f)$.
Fix $d \geq 1$, along with a primitive recursive $d$-tupling function $\Nb^d \imp \Nb$ with associated primitive recursive projections $p_0,\dots,p_{d-1} \colon \Nb \imp \Nb$.
Given a function $f$, define 
\[\Delta^f_{i,d}(x) = \begin{cases}
\Phi^f_{p_i(x)}(p_i(x)) & \text{if $\Phi^f_{p_j(x)}(p_j(x))\da$ for all $j \leq i$,} \\ 
\uparrow & \text{otherwise.}
\end{cases}\]
Write $D^f_{i+1,d} = \dom(\Delta^f_{i,d})$, and set $D^f_{0,d} = \Nb$.

\begin{Definition}\label{def-depth}
Consider a function $f \colon \Nb \imp \Nb$ and a length $d$ sequence $\vec g$ of functions $g_i \colon D^f_{i,d} \imp \Nb$ for $i < d$.
\begin{itemize}
\item The sequence $\vec g$ is \emph{depth $d$ diagonally non-recursive relative to $f$} ($\vec g$ is $\dnrstat_d(f)$ for short) if $(\forall x)(\exists i < d)(g_i(x) \neq \Delta^f_{i,d}(x))$.

\item The sequence $\vec g$ is \emph{$k$-bounded depth $d$ diagonally non-recursive relative to $f$} ($\vec g$ is $\dnrstat_d(k,f)$ for short) if it is $\dnrstat_d(f)$ and each $g_i$ is $k$-bounded.
\end{itemize}
\end{Definition}

Overloading notation as we did before, we let $\dnrstat_d(f)$ denote the formal statement ``there is a $\vec{g}$ that is $\dnrstat_d(f)$,'' and we let $\dnrstat_d(k,f)$ denote the formal statement ``there is a $\vec{g}$ that is $\dnrstat_d(k,f)$.''  Although different $d$-tupling schemes lead to different classes of $\dnrstat_d(f)$-functions, it is always possible to translate back and forth between any two such schemes.  In particular, the principles $\dnrstat_d(f)$ and $\dnrstat_d(k,f)$ are unaffected by such choices.  We therefore see that $\dnrstat(k,f)$ is equivalent to $\dnrstat_1(k,f)$ and that if $c > d$ then $\dnrstat_d(k,f)$ implies $\dnrstat_c(k,f)$.  A further relation between these principles is given by the following lemma.

\begin{Lemma}[$\rcasys$]\label{lem-packing}
Let $f \colon \Nb \imp \Nb$ be a function, and let $c,d,k \in \Nb$ be positive.  Then $\dnrstat_c(k^d,f)$ implies $\dnrstat_{cd}(k,f)$.  In particular, $\dnrstat(k^d,f)$ implies $\dnrstat_d(k,f)$.
\end{Lemma}

Schmerl considers colorings of graphs in classes $\cls K$ of a certain kind.  The class $\cls K$ is a \emph{universal class of graphs} if there is a set $K$ of finite (coded) graphs such that a graph $G$ belongs to $\cls{K}$ if and only if every finite induced subgraph of $G$ is isomorphic to a graph in $K$. The class $\cls K$ is a \emph{natural class of graphs} if it is moreover closed under disjoint sums. That is, if $G_0 = (V_0,E_0),G_1 = (V_1,E_1)$ are graphs in $\cls K$ with mutually disjoint vertex sets, then $G_0 + G_1 = (V_0 \cup V_1, E_0 \cup E_1)$ is also in $\cls K$.  It then follows that the class $\cls K$ is closed under countable disjoint sums.  For every positive integer $\ell$, the locally $\ell$-colorable graphs form a natural class of graphs.

The link between the on-line coloring games and the generalized $\dnrstat$ principles is the following result, which can be extracted from the proof of \cite{Schmerl:2000el}~Theorem~2.1.

\begin{Lemma}[$\rcasys$]\label{lem-Schmerl}
Let $f \colon \Nb \imp \Nb$ be a function, $\cls K$ be a natural class of graphs, and $d,k \in \Nb$ be positive.  If $\forall$ has a winning strategy in $\Gamma_d(\cls K,k)$ and $\dnrstat_d(k,f)$ fails, then there is an $f$-recursive graph from the class $\cls K$ that is not globally $k$-colorable.  We may further require that the connected components of this graph have size at most $d$, and the graph construction is uniform in the parameters $k,d$ and $f$.
\end{Lemma}

Lemma~\ref{lem-Schmerl} has two immediate consequences.

\begin{Theorem}[$\rcasys$]\label{thm-Schmerl++}
Let $\cls K$ be a natural class of graphs.
\begin{itemize}
\item For every $k \in \Nb$, if $\cls K$ is not locally on-line $k$-colorable but every graph in $\cls K$ is $k$-colorable, then $\exists d \forall f \dnrstat_d(k,f)$.

\item If $\cls K$ is not locally on-line $k$-colorable for any $k \in \Nb$ but every graph in $\cls K$ is finitely colorable, then $\forall f \exists d \exists k \dnrstat_d(k,f)$.
\end{itemize}
\end{Theorem}

It is provable in $\rcasys$ that the natural class of forests (i.e., graphs without cycles) is not locally on-line $k$-colorable for any $k$. More precisely, one can recursively construct a strategy for $\forall$ in the game $\Gamma_{2^k}(\cls K,k)$, where $\cls K$ is the class of forests.  Since forests are locally $2$-colorable, it follows that for every $\ell \geq 2$, the natural class of locally $\ell$-colorable graphs is likewise not locally on-line $k$-colorable for any $k$.

\begin{Corollary}\label{cor-COLtoDNRD}{\ }
\begin{itemize}
\item $\rcasys \vdash \forall k (\forall G\,\colstat(2,k,G) \imp \exists d\forall f\dnrstat_d(k,f))$.
\item $\rcasys \vdash \forall G\exists k\colstat(2,k,G) \imp \forall f\exists d \exists k\dnrstat_d(k,f))$.
\end{itemize}
\end{Corollary}

When working over $\rcasys + \ist$, both $\exists k \forall f \dnrstat(k,f)$ and $\forall f \exists k \dnrstat_d(k,f)$ are equivalent to $\wklstat$ by Theorem~\ref{thm-DNR2inISigma2}. A similar argument shows that $\exists d \exists k \forall f \dnrstat_d(k,f)$ and $\forall f \exists d \exists k \dnrstat_d(k,f)$ are likewise equivalent to $\wklstat$ over $\rcasys + \ist$. It follows from Corollary~\ref{cor-COLtoDNRD} that Conjecture~\ref{conj-gh} is true with $\rcasys + \ist$ in place of $\rcasys$.

\begin{Corollary}
$\rcasys + \ist \vdash  (\forall \ell \geq 2)(\forall k \geq \ell)(\wklstat \biimp \forall G\,\colstat(\ell,k,G))$.
\end{Corollary}

The relationships among diagonally non-recursive functions, depth $d$ diagonally non-recursive sequences, and graph colorings need further clarification.

\begin{Question}\label{q-DNRDvsCOL}{\ }
\begin{itemize}
\item Are $\exists k \exists d \forall f \dnrstat_d(k,f)$ and $\exists k\forall G\,\colstat(2,k,G)$ equivalent over $\rcasys$ (or over $\rcasys + \bst$)?
\item Are $\forall f \exists k \exists d \dnrstat_d(k,f)$ and $\forall G\exists k\colstat(2,k,G)$ equivalent over $\rcasys$ (or over $\rcasys + \bst$)?
\end{itemize}
\end{Question}

While $\dnrstat(k^d,f)$ implies $\dnrstat_d(k,f)$ over $\rcasys$ by Lemma~\ref{lem-packing}, it is not known whether the reverse implication holds.

\begin{Question}\label{q-DNRvsDNRD}{\ }
\begin{itemize}
\item Are $\exists k \forall f \dnrstat(k,f)$ and $\exists d \exists k \forall f \dnrstat_d(k,f)$ equivalent over $\rcasys$ (or over $\rcasys + \bst$)?
\item Are $\forall f \exists k \dnrstat(k,f)$ and $\forall f \exists d \exists k \dnrstat_d(k,f)$ equivalent over $\rcasys$ (or over $\rcasys + \bst$)?
\end{itemize}
\end{Question}

Similar to the diagonally non-recursive case, it is possible that $\forall \ell \forall G \exists k \colstat(\ell,k,G)$ is strictly weaker than $\forall \ell \exists k \forall G\, \colstat(\ell,k,G)$ over $\rcasys + \bst$.  However, our techniques do not readily adapt to avoiding graph colorings because the construction of an eventually $\dnrstat(b,h)$ function given only an upper bound on the index of a $\dnrstat(b,h)$ function in Lemma~\ref{lem-DNRblock} relies heavily on the homogeneity of diagonally non-recursive functions.  If $f_0$ and $f_1$ are diagonally non-recursive functions, then another diagonally non-recursive function $g$ can be obtained by choosing $g(n) \in \{f_0(n), f_1(n)\}$ for each $n$.  However, if $f_0$ and $f_1$ are graph colorings, there is no reason to expect that a $g$ chosen the same way is also a graph coloring.

\begin{Question}
Does $\forall \ell \forall G \exists k \colstat(\ell,k,G)$ imply $\forall \ell \exists k \forall G\, \colstat(\ell,k,G)$ over $\rcasys$ (or over $\rcasys+\bst$)?
\end{Question}

Motivated by Question~\ref{q-DNRDvsCOL}, we conclude by exploring further relationships between diagonally non-recursive functions and graph colorings.  First, we observe that the existence of $k$-bounded diagonally non-recursive functions does not suffice to ensure that locally $k$-colorable graphs are $(2k-1)$-colorable.

\begin{Proposition}
$\rcasys + \bst \nvdash \forall k(\forall f \dnrstat(k,f) \imp \forall G\, \colstat(k,2k-1,G))$.
\end{Proposition}

\begin{proof}
If $\rcasys + \bst \vdash \forall k(\forall f \dnrstat(k,f) \imp \forall G\, \colstat(k,2k-1,G))$, then also $\rcasys + \bst + \exists k \forall f \dnrstat(k,f) \vdash \exists k \forall G\, \colstat (k, 2k-1, G)$.  It would then follow from Theorem~\ref{thm-gh} that $\rcasys + \bst + \exists k \forall f \dnrstat(k,f) \vdash \wklstat$, contradicting Theorem~\ref{thm-noWKL}.  So $\rcasys + \bst \nvdash \forall k(\forall f \dnrstat(k,f) \imp \forall G\, \colstat(k,2k-1,G))$.
\end{proof}

A $(2k-1)$-coloring of a graph $G$ is also a $2k$-coloring of $G$, so asserting that every locally $k$-colorable graph $G$ is $2k$-colorable is potentially weaker than asserting that it is $(2k-1)$-colorable.  This situation raises the following question.

\begin{Question}\label{q-kp}
Does $\rcasys$ (or $\rcasys + \bst$) prove $\forall k(\forall f \dnrstat(k,f) \imp \forall G\, \colstat(k,2k,G))$?
\end{Question}

Note that because $\forall f \dnrstat(k,f)$ implies $\wklstat$ over $\rcasys$ for any fixed $k \in \omega$, the answer to the question is yes when restricted to a fixed $k \in \omega$.  Although we have not answered this question in general, we can formulate an analog of Theorem~\ref{thm-h} with $\exists k \forall f \dnrstat(k,f)$ replacing $\wklstat$ by restricting the class of graphs.

\begin{Definition}{\ }
\begin{itemize}
\item A \emph{complete $k$-partite graph} is a graph $G = (V,E)$ where $V$ is a set of vertices of the form $V = \{v_{(i,n)} : i < k \andd n \in \Nb \}$ and $E = \{(v_{(i,n)}, v_{(j,m)}) : i,j < k \andd n,m \in \Nb \andd i \neq j\}$.  

\item An \emph{ornamented complete $k$-partite graph} is a graph $G = (V \cup W, E)$, where $(V, E \cap (V \times V))$ is a complete $k$-partite graph and every $w \in W$ is either isolated or adjacent to exactly one $v \in V$.
\end{itemize}
\end{Definition}

\begin{Proposition}
\begin{align*}
\rcasys \vdash (\forall k \geq 2)(\forall f \dnrstat(k,f) \biimp \text{every ornamented complete $k$-partite graph is $k$-colorable}).
\end{align*}
\end{Proposition}

\begin{proof}
Fix $k \in \Nb$.

For the forward direction, let $G = (V \cup W, E)$ be an ornamented complete $k$-partite graph, where $V = \{v_{(i,n)} : i < k \andd n \in \Nb \}$ and $W = \{w_n : n \in \Nb\}$.  Define a function $h \colon \Nb \imp \Nb$ so that, for all $n, x \in \Nb$, $\Phi_{h(n)}^G(x) = i$ if there is an $m \in \Nb$ such that $w_n$ is adjacent to $v_{(i,m)}$ (and $\Phi_{h(n)}^G(x)\ua$ otherwise).  Let $g$ be $\dnrstat(k,G)$.  Define $\chi \colon V \imp k$ by $\chi(v_{(i,n)}) = i$ and $\chi(w_n) = g(h(n))$.  It is easy to verify that $\chi$ is a $k$-coloring of $G$.

For the backward direction, let $G_0 = (V,E_0)$ be a complete $k$-partite graph, and, given $f$, extend $G_0$ to the ornamented complete $k$-partite graph $G = (V \cup W, E)$, where $W = \{w_n : n \in \Nb\}$, by defining $(w_n, v_{(i,s)}) \in E$ if and only if $\Phi^f_{n,s}(n) = i$ and $(\forall t < s)(\Phi^f_{n,t}(n)\ua)$.  Let $\chi$ be a $k$-coloring of $G$, and permute the colors so that $\chi(v_{(i,0)}) = i$ for each $i < k$.  Then the function $g$ defined by $g(n) = \chi(w_n)$ is $\dnrstat(k,f)$.
\end{proof}

Say that a graph $G_0 = (V_0,E_0)$ embeds into a graph $G_1 = (V_1,E_1)$ if there is an injection $h \colon V_0 \imp V_1$ such that $(\forall v, w \in V_0)((v,w) \in E_0 \imp (h(v),h(w)) \in E_1)$.  Notice that a graph is $k$-colorable if and only if it embeds into a complete $k$-partite graph.  In fact, it is not hard to see that $\rcasys$ proves this fact.  We can rephrase Question~\ref{q-kp} in terms of embeddings as follows.

\begin{Question}
Does $\rcasys$ (or $\rcasys + \bst$) prove the following statement?
\begin{align*}
\forall k(\forall f \dnrstat(k,f) \imp &\text{ every locally $k$-colorable graph}\\
&\text{ can be embedded into an ornamented complete $2k$-partite graph})
\end{align*}
\end{Question}

\section*{Acknowledgement}

We thank our anonymous referee for his or her very helpful report.

\bibliographystyle{amsplain}
\bibliography{existskDNRkvsWKL}

\begin{bibdiv}
\begin{biblist}

\bib{AmbosSpies:2004cv}{article}{
      author={Ambos-Spies, Klaus},
      author={Kjos-Hanssen, Bj{\o}rn},
      author={Lempp, Steffen},
      author={Slaman, Theodore~A.},
       title={{Comparing DNR and WWKL}},
        date={2004},
     journal={Journal of Symbolic Logic},
      volume={69},
      number={4},
       pages={1089\ndash 1104},
         url={http://dx.doi.org/10.2178/jsl/1102022212},
}

\bib{Austen}{book}{
      author={Austen, Jane},
       title={{Pride and Prejudice}},
   publisher={T. Egerton, Whitehall},
        date={1813},
}

\bib{Bean:1976ta}{article}{
      author={Bean, Dwight~R.},
       title={{Effective coloration}},
        date={1976},
     journal={Journal of Symbolic Logic},
      volume={41},
      number={2},
       pages={469\ndash 480},
         url={http://www.jstor.org/stable/10.2307/2272247},
}

\bib{Chong:2013wx}{article}{
      author={Chong, C.~T.},
      author={Li, Wei},
      author={Yang, Yue},
       title={{Nonstandard models in recursion theory and reverse
  mathematics}},
        date={2014},
     journal={Bulletin of Symbolic Logic},
      volume={20},
      number={2},
       pages={170\ndash 200},
}

\bib{Chong:1992uz}{article}{
      author={Chong, C.~T.},
      author={Mourad, K.~J.},
       title={{$\Sigma_n$ definable sets without $\Sigma_n$ induction}},
    language={English},
        date={1992-11},
     journal={Transactions of the American Mathematical Society},
      volume={334},
      number={1},
       pages={349\ndash 363},
         url={http://www.jstor.org/stable/2153986},
}

\bib{Chong:2001iv}{article}{
      author={Chong, C.~T.},
      author={Qian, Lei},
      author={Slaman, Theodore~A.},
      author={Yang, Yue},
       title={{$\Sigma_2$ induction and infinite injury priority arguments,
  part III: Prompt sets, minimal pairs and Shoenfield's conjecture}},
        date={2001},
     journal={Israel Journal of Mathematics},
      volume={121},
       pages={1\ndash 28},
         url={http://dx.doi.org/10.1007/BF02802493},
}

\bib{Chong:2012kz}{article}{
      author={Chong, C.~T.},
      author={Slaman, Theodore~A.},
      author={Yang, Yue},
       title={{$\Pi^1_1$-conservation of combinatorial principles weaker than
  Ramsey's theorem for pairs}},
        date={2012},
     journal={Advances in Mathematics},
      volume={230},
      number={3},
       pages={1060\ndash 1077},
         url={http://dx.doi.org/10.1016/j.aim.2012.02.025},
}

\bib{Chong:2013prep}{unpublished}{
      author={Chong, C.~T.},
      author={Slaman, Theodore~A.},
      author={Yang, Yue},
       title={{The inductive strength of Ramsey's theorem for pairs}},
        date={2013},
        note={in preparation},
}

\bib{Chong:2012vp}{article}{
      author={Chong, C.~T.},
      author={Slaman, Theodore~A.},
      author={Yang, Yue},
       title={{The metamathematics of stable Ramsey's theorem for pairs}},
        date={2014},
     journal={Journal of the American Mathematical Society},
      volume={27},
      number={3},
       pages={863\ndash 892},
}

\bib{Chong:1997he}{article}{
      author={Chong, C.~T.},
      author={Yang, Yue},
       title={{$\Sigma_2$ induction and infinite injury priority arguments,
  Part II Tame $\Sigma_2$ coding and the jump operator}},
        date={1997},
     journal={Annals of Pure and Applied Logic},
      volume={87},
      number={2},
       pages={103\ndash 116},
         url={http://dx.doi.org/10.1016/0168-0072(96)00028-0},
}

\bib{Chong:1998tu}{article}{
      author={Chong, C.~T.},
      author={Yang, Yue},
       title={{$\Sigma_2$ induction and infinite injury priority argument, Part
  I: Maximal sets and the jump operator}},
        date={1998},
     journal={Journal of Symbolic Logic},
      volume={63},
      number={3},
       pages={797\ndash 814},
  url={http://scholar.google.com/scholar?q=related:m8zWE8VqMioJ:scholar.google.com/&hl=en&num=30&as_sdt=0,5},
}

\bib{Chubb:2009gr}{article}{
      author={Chubb, Jennifer},
      author={Hirst, Jeffry~L.},
      author={McNicholl, Timothy~H.},
       title={{Reverse mathematics, computability, and partitions of trees}},
        date={2009},
     journal={Journal of Symbolic Logic},
      volume={74},
      number={1},
       pages={201\ndash 215},
         url={http://dx.doi.org/10.2178/jsl/1231082309},
}

\bib{Corduan:2010ig}{article}{
      author={Corduan, Jared},
      author={Groszek, Marcia~J.},
      author={Mileti, Joseph~R.},
       title={{Reverse mathematics and Ramsey's property for trees}},
        date={2010},
     journal={Journal of Symbolic Logic},
      volume={75},
      number={3},
       pages={945\ndash 954},
         url={http://dx.doi.org/10.2178/jsl/1278682209},
}

\bib{Friedman:1975wy}{inproceedings}{
      author={Friedman, Harvey},
       title={{Some systems of second order arithmetic and their use}},
        date={1975},
   booktitle={Proceedings of the {I}nternational {C}ongress of {M}athematicians
  ({V}ancouver, {B}. {C}., 1974), {V}ol. 1},
   publisher={Canad. Math. Congress, Montreal, Que.},
       pages={235\ndash 242},
}

\bib{Gasarch:1998gy}{article}{
      author={Gasarch, William},
      author={Hirst, Jeffry~L.},
       title={{Reverse mathematics and recursive graph theory}},
        date={1998},
     journal={Mathematical Logic Quarterly},
      volume={44},
      number={4},
       pages={465\ndash 473},
         url={http://www.ams.org/mathscinet-getitem?mr=MR1654281},
}

\bib{Groszek:1996kq}{article}{
      author={Groszek, Marcia~J.},
      author={Mytilinaios, Michael~E.},
      author={Slaman, Theodore~A.},
       title={{The Sacks density theorem and $\Sigma_2$-bounding}},
        date={1996},
     journal={Journal of Symbolic Logic},
      volume={61},
      number={2},
       pages={450\ndash 467},
         url={http://www.ams.org/mathscinet-getitem?mr=MR1394609},
}

\bib{Groszek:1994wa}{article}{
      author={Groszek, Marcia~J.},
      author={Slaman, Theodore~A.},
       title={{On Turing reducibility}},
        date={1994},
     journal={preprint},
}

\bib{Hajek:1993vb}{incollection}{
      author={H{\'a}jek, Petr},
       title={{Interpretability and fragments of arithmetic}},
        date={1993},
   booktitle={Arithmetic, proof theory, and computational complexity},
      editor={Clote, Peter},
      editor={Kraj{\'\i}cek, Jan},
      series={Oxford Logic Guides},
      volume={23},
   publisher={Oxford University Press},
       pages={185\ndash 196},
}

\bib{Hajek:1998uh}{book}{
      author={H{\'a}jek, Petr},
      author={Pudl{\'a}k, Pavel},
       title={{Metamathematics of First-Order Arithmetic}},
      series={Perspectives in Mathematical Logic},
   publisher={Springer-Verlag},
     address={Berlin},
        date={1998},
        ISBN={3-540-63648-X},
}

\bib{Hirst:1987}{incollection}{
      author={Hirst, Jeffry~L.},
       title={Marriage theorems and reverse mathematics},
        date={1990},
   booktitle={Logic and computation ({P}ittsburgh, {PA}, 1987)},
      series={Contemporary Mathematics},
      volume={106},
   publisher={American Mathematical Society},
     address={Providence, RI},
       pages={181\ndash 196},
}

\bib{JockuschJr:1989vs}{article}{
      author={Jockusch, Carl~G., Jr.},
       title={{Degrees of functions with no fixed points}},
        date={1989},
     journal={Logic, Methodology and Philosophy of Science VIII},
       pages={191\ndash 201},
  url={http://books.google.com/books?hl=en&lr=&id=JvErejvds5oC&oi=fnd&pg=PA191&dq=degrees+of+functions+with+no+fixed+points&ots=mUSaU63tI_&sig=hSkk5lc_uLnkNP8NfkUFaT6u5bY},
}

\bib{JockuschJr:1972wma}{article}{
      author={Jockusch, Carl~G., Jr.},
      author={Soare, Robert~I.},
       title={{$\Pi^0_1$ classes and degrees of theories}},
        date={1972},
     journal={Transactions of the American Mathematical Society},
      volume={173},
       pages={33\ndash 56},
         url={http://www.ams.org/mathscinet-getitem?mr=MR0316227},
}

\bib{Mytilinaios:1989ba}{article}{
      author={Mytilinaios, Michael~E.},
       title={{Finite injury and $\Sigma_1$-induction}},
        date={1989},
     journal={Journal of Symbolic Logic},
      volume={54},
      number={1},
       pages={38\ndash 49},
         url={http://www.ams.org/mathscinet-getitem?mr=MR987320},
}

\bib{Schmerl:2000el}{article}{
      author={Schmerl, James~H.},
       title={{Graph coloring and reverse mathematics}},
        date={2000},
     journal={Mathematical Logic Quarterly},
      volume={46},
      number={4},
       pages={543\ndash 548},
         url={http://www.ams.org/mathscinet-getitem?mr=MR1791549},
}

\bib{Schmerl:2005wr}{incollection}{
      author={Schmerl, James~H.},
       title={{Reverse mathematics and graph coloring: eliminating
  diagonalization}},
        date={2005},
   booktitle={Reverse mathematics 2001},
   publisher={Assoc. Symbol. Logic},
     address={La Jolla, CA},
       pages={331\ndash 348},
         url={http://www.ams.org/mathscinet-getitem?mr=MR2185444},
}

\bib{SimpsonTalk}{misc}{
      author={Simpson, Stephen~G.},
       title={Why the recursion theorists ought to thank me},
         how={{Talk given at the Annual Meeting of the American Philosophical
  Association}},
        date={2001},
}

\bib{Simpson:2009vv}{book}{
      author={Simpson, Stephen~G.},
       title={{Subsystems of Second Order Arithmetic}},
   publisher={Cambridge University Press},
        date={2009},
}

\bib{Slaman:2004bh}{article}{
      author={Slaman, Theodore~A.},
       title={{$\Sigma_n$-Bounding and $\Delta_n$-Induction}},
    language={English},
        date={2004-08},
     journal={Proceedings of the American Mathematical Society},
      volume={132},
      number={8},
       pages={2449\ndash 2456},
         url={http://dx.doi.org/10.1090/S0002-9939-04-07294-6},
}

\bib{Slaman:1989fc}{incollection}{
      author={Slaman, Theodore~A.},
      author={Woodin, W.~Hugh},
       title={{$\Sigma_1$-Collection and the finite injury priority method}},
        date={1989},
   booktitle={Mathematical logic and applications},
      editor={Shinoda, Juichi},
      editor={Tugu{\'e}, Tosiyuki},
      editor={Slaman, Theodore~A.},
      series={Lecture Notes in Mathematics},
      volume={1388},
   publisher={Springer Berlin Heidelberg},
       pages={178\ndash 188},
         url={http://dx.doi.org/10.1007/BFb0083670},
}

\bib{Soare:1987vq}{book}{
      author={Soare, Robert~I.},
       title={{Recursively Enumerable Sets and Degrees}},
      series={Perspectives in Mathematical Logic},
   publisher={Springer-Verlag},
     address={Berlin},
        date={1987},
        ISBN={3-540-15299-7},
  url={http://www.worldcat.org/title/recursively-inumerable-sets/oclc/155020285},
}

\bib{Yu:1990gl}{article}{
      author={Yu, Xiaokang},
      author={Simpson, Stephen~G.},
       title={{Measure theory and weak K\"onig's lemma}},
        date={1990},
     journal={Archive for Mathematical Logic},
      volume={30},
      number={3},
       pages={171\ndash 180},
         url={http://dx.doi.org/10.1007/BF01621469},
}

\end{biblist}
\end{bibdiv}

\end{document}